\documentclass[10pt,reqno]{article}

\usepackage{amsthm,amsmath,amssymb,amsbsy,amsfonts,graphics,epsfig,graphicx,mathdots,cleveref}

\newcommand{\R}{{\mathbb R}}
\newcommand{\C}{{\mathbb C}}

\newcommand{\Z}{{\mathbb Z}}
\newcommand{\N}{{\mathbb N}}
\newcommand{\Oo}{{\mathcal O}}

\newcommand{\beq}{\begin{equation}}
\newcommand{\eeq}{\end{equation}}

\newtheorem{theo}{{\bf Theorem}}[section]
\newtheorem{cor}[theo]{{\bf Corollary}}
\newtheorem{lem}[theo]{{\bf Lemma}}
\newtheorem{Prop}[theo]{{\bf Proposition}}

\newcounter{definition}
\newenvironment{defn}[1][]{\refstepcounter{definition}\par\medskip
	\noindent \textbf{Definition~\thedefinition. #1} \rmfamily}{\medskip}

\newcounter{example}
\newenvironment{ex}[1][]{\refstepcounter{example}\par\medskip
   \noindent \textbf{Example~\theexample. #1} \rmfamily}{\medskip}
   
\newcounter{remark}   
\newenvironment{rem}[1][]{\refstepcounter{remark}\par\medskip
   \noindent \textbf{Remark~\theremark. #1} \rmfamily}{\medskip}

\numberwithin{equation}{section}

\author{A. Bourget and T. McMillen\footnote{\textbf{Mailing address}: Department of Mathematics, California State University (Fullerton),
McCarthy Hall 154, Fullerton CA 92834 (US). \textbf{Email address}: abourget@fullerton.edu} \\ Department of Mathematics \\ California State University, Fullerton}

\title{Stability of Doubly Indexed Discrete Linear Systems}

\date{}

\begin{document}

\maketitle

\begin{abstract}
	We study the stability of the two-dimensional discrete dynamical system defined by 
	\[ x_{i,j} = A x_{i-1,j} + B x_{i,j-1},\]
	where $A, B \in \mathcal{M}_n(\mathbb{C})$.  Although this system—commonly known as the Fornasini--Marchesini model---has been examined for decades, several key stability results have lacked rigorous proofs. This paper establishes precise and self-contained criteria for Lyapunov, asymptotic, and neutral stability under general bounded boundary conditions. In particular, we prove that asymptotic stability is equivalent to the spectral condition $ \rho(z_1 A + z_2 B) < 1 $ for all $ |z_1| \leq 1,  |z_2|  \leq 1 $. We further analyze the commutative case, derive recurrence relations for anti-diagonal sums, and present new examples and counterexamples illustrating the nuances of stability. The results are complemented by sufficient conditions based on matrix norms, Lyapunov functions, and the joint spectral radius.
\end{abstract}



\section{Introduction}

Let $\mathcal{M}_n(\C)$ denote the set of all $ n \times n $ complex matrices, and let $ A, B \in \mathcal{M}_n(\C)$. We consider the doubly indexed discrete dynamical system defined by the recurrence relation
\begin{equation} \label{2D FM}
x_{i,j} = A x_{i-1,j} + B x_{i,j-1} \qquad (i, j \in \mathbb{N}),
\end{equation}
subject to prescribed boundary conditions $ \{x_{i,0}\} $ and $ \{x_{0,j}\} $. This two-dimensional recurrence defines a system whose evolution proceeds along two independent directions and is thus commonly referred to as a 2D system.

The study of such systems originated in the 1970s with the pioneering work of Fornasini and Marchesini \cite{FoMa78, FoMa80}, who introduced state-space models for multidimensional systems. These so-called Fornasini--Marchesini (FM) systems have since found widespread applications in engineering and applied mathematics, including digital signal processing, image analysis, and multidimensional control~\cite{Bo82,Ka85, GaRo18, TiDh12}. 

Their mathematical structure generalizes the classical one-dimensional system $x_{k+1}=Ax_k$, whose asymptotic behavior is completely characterized by the spectral radius $\rho(A)$ of the transition matrix  $A$.  In the 1D setting, the classification of stability is well established:
\begin{itemize}
\item $x_k \rightarrow 0$ if $\rho(A)<1$, 
\item $x_k$ is bounded but does not tend to zero for some $x_0$ if $\rho(A)=1$ and all eigenvalues with $|\lambda|=1$ are semisimple, 
\item $\|x_k\| \rightarrow \infty$ for some $x_0$ if $\rho(A)>1$, or if $\rho(A) = 1$ and $A$ has an eigenvalue with $|\lambda|=1$ that is not semisimple. 
\end{itemize}

By contrast, the analysis of stability in 2D systems is significantly more intricate. One major complication arises from the influence of boundary data along two directions, which interact in a nontrivial way through the recurrence. Moreover, different notions of convergence  can lead to different classifications of stability. The literature often assumes special classes of boundary conditions—for example, sequences that vanish eventually or decay at infinity—leading to simplifications that obscure the full complexity of the system’s behavior.

We take a general, foundational approach: we assume only bounded boundary data and use Pringsheim convergence (i.e., $x_{i,j} \to 0$ as $i,j \to \infty$ independently). This setting cleanly separates Lyapunov, asymptotic, and neutral stability, and makes explicit how these notions depend on the chosen boundary class and mode of convergence.
 
Our primary objective is to provide a complete and rigorous framework for analyzing the stability of 2D systems, filling gaps that persist in prior treatments. While some aspects of this system have been studied in the past, many arguments lack full mathematical rigor, and certain key results are stated without proof or only under restrictive assumptions. In particular, to the best of our knowledge, the spectral radius condition

\begin{equation*}  
\rho(z_1A+z_2B) < 1 \qquad \text{for all } |z_1| \leq 1,  |z_2| \leq  1
\end{equation*}
does not appear to have been stated in the literature as a necessary and sufficient condition for the asymptotic stability of the origin. This condition, which generalizes the 1D spectral criterion to the 2D setting, plays a central role in our analysis and leads to a precise characterization of the system's long-term behavior.

We also investigate several structural features of the system, including closed-form solutions via generalized matrix polynomials, recurrence relations along anti-diagonals, and norm-based and joint spectral radius-based sufficient conditions for stability. Our results apply in both the commutative case (when  $A$ and $B$ commute) and the noncommutative case, with sharp criteria for when the origin is asymptotically stable, neutrally stable, or unstable. We provide concrete examples—some new, others correcting existing misconceptions in the literature—that illustrate the delicate interplay between the spectral properties of $A$ and $B$ and the system’s global behavior.

A key theme throughout the paper is the contrast between the commutative and noncommutative settings. When $A$ and $B$ commute, the dynamics can often be reduced, after a similarity transformation, to lower-dimensional subsystems whose behavior is governed by paired spectral data; in the diagonalizable case, $A$ and $B$ are simultaneously diagonalizable, and stability reduces to scalar conditions on the corresponding eigenvalue pairs. More generally, commutativity yields a common invariant decomposition (cf.\ Lemma~4.3), which effectively decouples the analysis into blocks. In the noncommutative case, no such simultaneous reduction is available: eigendirections need not be compatible, and stability is controlled by the joint behavior of the pencil $z_1A+z_2B$ over the torus $|z_1|=|z_2|=1$. This distinction explains why the commutative case admits a more explicit classification, whereas the general case requires genuinely multidimensional spectral criteria.

The structure of the paper is as follows:
\begin{itemize}
	\item In Section~2, we introduce the stability notions used throughout the paper and fix the functional setting: bounded boundary data together with Pringsheim convergence in the interior of the lattice. We also explain how this setting differs from the commonly used $c_0$-boundary framework and from convergence of the whole array as $i+j\to\infty$.
	
    \item Section~3 establishes an explicit representation formula for the solution $x_{i,j}$ in terms of the boundary data and the noncommutative word sums $W_{i,j}(A,B)$. This formula is used to relate stability to the two one-axis contributions from the boundary data. We also derive a recurrence for the anti-diagonal sums, which yields a simple instability criterion when $\rho(A+B)>1$.
   \item In Section 4, we analyze the commutative case in detail. Compared with earlier work on asymptotic stability for commuting Fornasini--Marchesini models, we do not assume that one of the matrices is diagonalizable, and we obtain a full classification into asymptotic, neutral, and unstable regimes in the bounded-boundary/Pringsheim setting.
   
   \item Section 5 treats the general noncommutative case, includes a rigorous proof of the main spectral condition for asymptotic stability, and presents several sufficient conditions for Lyapunov and asymptotic stability using matrix norms, numerical radius, Lyapunov functions, and the joint spectral radius.
   \item Section 6 discusses BIBO-type estimates for $\ell^1$ boundary data.
  \end{itemize}
   
Throughout the paper, we aim to develop a rigorous and accessible framework that clarifies the behavior of this rich and subtle class of dynamical systems. We conclude in Section 7 with remarks on natural generalizations, including multi-directional systems and the inclusion of diagonal terms, and outline several directions for future research.

\section{Definitions}

In this section, we introduce the stability notions used throughout the paper. Our formulation follows the standard approach in dynamical systems (see, e.g., Strogatz~\cite{st94}), adapted to the doubly indexed setting and to boundary data prescribed along the coordinate axes.

Let $A,B\in \mathcal{M}_n(\C)$. We denote by $\mathcal{S}(A,B)$ the two-dimensional discrete dynamical system
\begin{equation} \label{S(A,B)}
	x_{i,j}=Ax_{i-1,j}+Bx_{i,j-1}, \qquad (i,j\in\N),
\end{equation}
together with boundary conditions $\{x_{i,0}\}_{i\ge 1}$ and $\{x_{0,j}\}_{j\ge 1}$.
Throughout the paper we assume
\[  \{x_{i,0}\}_{i\ge 1},\ \{x_{0,j}\}_{j\ge 1}\in \ell^\infty(\N;\C^n). \]
The vector $x_{0,0}$ may be chosen arbitrarily -- unless otherwise stated, we set $x_{0,0}=0$.    

Given such boundary data, the recurrence \eqref{S(A,B)} determines a unique array $\{x_{i,j}\}_{i,j\ge 0}$ by forward propagation on the lattice. In the present paper, we work primarily with bounded boundary data and use Pringsheim convergence for the interior trajectory: the statement $x_{i,j}\to 0$ as $i,j\to\infty$ means that for every $\varepsilon>0$ there exists $N\in\N$ such that
\[ \|x_{i,j}\|<\varepsilon \qquad \text{whenever } i\ge N \text{ and } j\ge N.\]

Throughout the paper, we use the following notation for the torus and the closed bidisk:
\[ \mathbb T=\{z\in\mathbb C: |z|=1\}, \qquad \mathbb T^2=\mathbb T\times\mathbb T, \]
and
\[ \overline{\mathbb D}=\{z\in\mathbb C: |z|\le 1\}, \qquad \overline{\mathbb D}^{\,2} =  \overline{\mathbb D} \times  \overline{\mathbb D}. \]

Unless stated otherwise, stability of $\mathcal{S}(A,B)$ refers to stability of the fixed point $x=0$ relative to this boundary class and this mode of convergence.

\begin{defn}  \label{def:stability}
	The system ${\cal S}(A,B)$ is said to be:
	
	\begin{enumerate}
		\item[(a)] \emph{Lyapunov stable} if for every $\varepsilon>0$, there exists $\delta = \delta(\varepsilon)>0$ such that the corresponding solution satisfies
		\[ \|x_{i,j}\|<\varepsilon \qquad \text{for all } i,j\in\mathbb N \]
		whenever 
		\[ \sup_{i\geq 1}\|x_{i,0}\|<\delta \quad\text{and}\quad \sup_{j\geq 1}\|x_{0,j}\|<\delta . \]

		\item[(b)] \emph{Attractive} if there exists $\delta>0$ such that, for each
		admissible boundary data satisfying
		\[ \sup_{i\geq 1}\|x_{i,0}\|<\delta \quad\text{and}\quad \sup_{j\geq 1}\|x_{0,j}\|<\delta ,\]
		the corresponding solution satisfies
		\[ \lim_{i,j\to\infty} x_{i,j}=0 \]
		in the sense of Pringsheim convergence. That is,  for every $\varepsilon>0$, there exists an integer $N$, possibly depending on the boundary data and on
		$\varepsilon$, such that
		\[ \|x_{i,j}\|<\varepsilon \qquad \text{whenever } i\geq N \text{ and } j\geq N . \]
			
		\item[(c)] \emph{Asymptotically stable} if it is Lyapunov stable and attractive.
		
		\item[(d)] \emph{Neutrally stable} if it is Lyapunov stable but not attractive.
		
		\item[(e)] \emph{Unstable} if it is not Lyapunov stable.
	\end{enumerate}
\end{defn}

\begin{rem} We make several comments regarding the above definitions and the choice of framework for the system $\mathcal{S}(A,B)$. 
\begin{enumerate}
	\item The norm $\|\cdot\|$ in Definition~\ref{def:stability} is an arbitrary norm on $\C^n$; since all norms on $\C^n$ are equivalent, the stability notions are independent of this choice. On occasion, it will be convenient to use the max norm
	\[ 	\|x\|_\infty=\max\{|x_k|:k=1,\dots,n\}. \]
	All matrix norms are assumed to be induced (and hence submultiplicative).
	
	 \item The stability notions considered here should not be interpreted as properties of a single prescribed boundary data. Rather, they are properties of the system relative to a specified class of admissible boundary data and a specified regime of convergence at infinity. Once this functional setting is fixed, the property is determined by the matrices defining the system; however, changing the admissible boundary-data class or the convergence regime may change the resulting stability condition.  See Remark~\ref{rem4} for further discussion. 
    
    \item The convergence requirement used in the present paper is Pringsheim convergence, which focuses on the interior asymptotic regime $x_{i,j}\to 0$ as both indices tend to infinity. This is weaker than requiring  convergence of the whole array as $i+j\to\infty$, a condition often used in the literature, for instance in $c_0$-type settings~\cite{Ba18, Ba19,BaClRiSiYe23, Ba23,BaCl23}. The two notions are not directly comparable unless both the admissible boundary class and the mode of convergence are specified. In particular, if one requires $x_{i,j}\to 0$ as $i+j\to\infty$ for the whole array, then non-vanishing boundary data are incompatible with attractivity, since the boundary values themselves are part of the array. Thus, that stronger convergence regime is naturally paired with boundary data that vanish at infinity, for instance compactly supported data or, more generally, data in $c_0$. By contrast, Pringsheim convergence is well suited to the bounded-boundary setting adopted here, since it isolates the interior behavior  of the solution from possible non-decay along the coordinate axes.

    \item Fornasini and Marchesini~\cite{FoMa78, FoMa80} investigated related 2D systems under both forward and backward evolution (with indices in $\Z$ rather than $\N$), and considered boundary conditions prescribed along a diagonal (e.g. $\{x_{-r,r}\}$) rather than along the axes.
  
	\item By choosing boundary conditions supported on a single axis, it is straightforward to verify that if either of the one-dimensional systems associated with $A$ or $B$ is unstable, then the two-dimensional system $\mathcal{S}(A,B)$ is also unstable. For instance, assume that $x_{i,0}=0$ for all $i\ge 1$. Then
	\[  x_{i+1,1}=Ax_{i,1}, \]
	which shows that $\mathcal{S}(A,B)$ restricted to the first column reduces to the one-dimensional system associated with $A$. Similarly, imposing the boundary conditions $x_{0,j}=0$ for all $j\ge 1$, we obtain
	\[ x_{1,j+1}=Bx_{1,j}, \]
	which corresponds to the one-dimensional system associated with $B$. As several examples later in the paper show, however, the asymptotic or Lyapunov stability of $\mathcal{S}(A,B)$ cannot, in general, be inferred from the stability properties of the corresponding one-dimensional systems.
	\end{enumerate}
\end{rem}

\begin{rem} \label{rem2}
	Throughout the paper we will use two elementary invariance properties: the stability type of $\mathcal{S}(A,B)$ is unchanged under unimodular rescaling and similarity transformations.
\begin{enumerate}
	\item  	If $(z_1,z_2)\in\mathbb{T}^2$, then $\mathcal{S}(A,B)$ and $\mathcal{S}(z_1A,z_2B)$ have the same stability type. Indeed, if $\{x_{i,j}\}$ is a solution of $\mathcal{S}(A,B)$, define
	\[ x'_{i,j}=z_1^i z_2^j x_{i,j}. \]
	Then $\{x'_{i,j}\}$ satisfies
	\[ x'_{i,j}=z_1A x'_{i-1,j}+z_2B x'_{i,j-1}, \]
	with bounded boundary data
	\[ x'_{i,0}=z_1^i x_{i,0}, \qquad x'_{0,j}=z_2^j x_{0,j}. \]
	
	\item If $X$ is invertible, then $\mathcal{S}(A,B)$ and 
	\[\mathcal{S}\left(XAX^{-1},XBX^{-1}\right) \] 
	have the same stability type. Indeed, $y_{i,j} = Xx_{i,j}$ satisfies
	\[ y_{i,j} = XAX^{-1} y_{i-1,j} + XBX^{-1}y_{i,j-1}.\]
	In both cases, the transformation is bijective and preserves bounded boundary data as well as convergence, boundedness, and unboundedness of solutions. Hence the stability type is unchanged.
\end{enumerate}
\end{rem}

\section{Preliminary results}

We begin with an explicit representation formula for the solution $x_{i,j}$ in terms of the boundary data on the two coordinate axes. Related formulas appear in Proposition 2.1 of~\cite{Ba23} and Theorem 5 of~\cite{FoMa80}. Our purpose here is to record the precise form needed below: the formula is written directly in terms of the boundary sequences $\{x_{i,0}\}$ and $\{x_{0,j}\}$, and the coefficients are expressed as the noncommutative word sums $W_{i,j}(A,B)$.

\begin{Prop} \label{BigForm}
	For every $i,j \in \mathbb{N}$, the solution of \eqref{S(A,B)} satisfies
	\begin{equation} \label{BigFormEqn}
		x_{i,j} = \sum_{k=0}^{j-1} W_{i-1,k}(A,B)\, A x_{0,j-k} + \sum_{k=0}^{i-1} W_{k,j-1}(A,B)\, B x_{i-k,0},
	\end{equation}
	where $W_{i,j}(A,B)$ denotes the sum of all words in the letters $A$ and $B$ containing exactly $i$ copies of $A$ and $j$ copies of $B$. If $A$ and $B$ commute, then
	\[ W_{i,j}(A,B)=\binom{i+j}{i}A^iB^j. \]
\end{Prop}

\begin{proof}
	For simplicity, we denote $W_{i,j}(A,B)$ by $W_{i,j}$ throughout the proof. We work first in the ring of formal power series in $z_1,z_2$ with coefficients in $\C^n$. Thus, the following generating-function manipulations are purely algebraic, and no convergence is being assumed. Define formally 
	\begin{equation} \label{genfun1}
	X(z_1, z_2) = \sum_{i,j \geq 0} x_{i,j} z_1^i z_2^j
	\end{equation}
	Note that
	\begin{align*} 
		z_1 AX & = \sum_{i,j \geq 0} A x_{i,j} z_1^{i+1} z_2^j = \sum_{i\geq 1, j\geq 0} Ax_{i-1,j} z_1^i z_2^j	\\
		z_2 B X & = \sum_{i,j \geq 0} Bx_{i,j} z_1^i z_2^{j+1} = \sum_{i\geq 0, j\geq 1} Bx_{i,j-1} z_1^i z_2^j
	\end{align*}
	It then follows from \eqref{S(A,B)} that 
	\begin{align} \label{U and V}
	X - z_1 A X- z_2 B X & =  \sum_{i \geq 1} (x_{i,0}-Ax_{i-1,0}) z_1^i + \sum_{j \geq 1} (x_{0,j}-Bx_{0,j-1}) z_2^j   \nonumber \\
       & =: U(z_1) + V(z_2) 
    \end{align}   	
Since the constant term of $I-z_1A-z_2B$ is $I$, it is invertible in the formal power series ring, with inverse interpreted formally as the Neumann series
\[ (I-z_1 A -z_2 B)^{-1} =  \sum_{n \geq 0} (z_1A+z_2B)^n. \] 
As a consequence, 
\begin{equation} \label{X(z1,z2)}
  X(z_1,z_2) = (I -z_1 A -z_2 B)^{-1} (U(z_1) + V(z_2)). 
\end{equation}

We multiply the formal power series for the resolvent of $z_1 A +z_2 B$ and $U(z_1)$ using the Cauchy product to obtain 
\begin{align} \label{U}
	(I -z_1 A -z_2 B)^{-1} U(z_1) & = \sum_{n \geq 0} (z_1 A +z_2 B)^n \ \sum_{i \geq 1} (x_{i,0}-Ax_{i-1,0}) \, z_1^i \nonumber \\
	  & = \sum_{i,j \geq 0} W_{i,j} z_1^i z_2^j \  \sum_{i \geq 1} (x_{i,0}-Ax_{i-1,0}) \, z_1^i \nonumber \\
	  & = \sum_{i,j \geq 0} \sum_{k=1}^i W_{i-k,j}(x_{k,0}-Ax_{k-1,0}  ) z_1^i z_2^j
\end{align}
By a similar argument applied to  $V(z_2)$, we obtain
\begin{align} \label{V}
	(I -z_1 A -z_2 B)^{-1} V(z_2) & =  \sum_{i,j \geq 0} \sum_{k=1}^j W_{i,j-k}(x_{0,k}-Bx_{0,k-1}  ) z_1^i z_2^j.
\end{align}

Comparing the coefficients of $z_1^i z_2^j$ in the expressions  \eqref{genfun1}, \eqref{U} and \eqref{V}, it follows that
\begin{align} \label{xij}
	x_{i,j} = \sum_{k=1}^i W_{i-k,j}(x_{k,0}-Ax_{k-1,0}  ) +  \sum_{k=1}^j W_{i,j-k}(x_{0,k}-Bx_{0,k-1}  )
\end{align}
We now use the combinatorial identity 
\[  W_{i,j} = W_{i-1,j} \, A + W_{i,j-1} \, B \] 
in the first sum above to obtain
\begin{align} \label{xij1}
	\sum_{k=1}^i  W_{i-k,j} (x_{k,0}-Ax_{k-1,0} ) & = \sum_{k=1}^i \left(  W_{i-k,j} x_{k,0} - W_{i-k,j} A x_{k-1,0} \right) \nonumber \\
	   & = \sum_{k=1}^i  W_{i-k,j} x_{k,0}  - \sum_{k=1}^i W_{i-k+1,j} x_{k-1,0}  \nonumber  \\
	   & \quad + \sum_{k=1}^i W_{i-k+1,j-1} B x_{k-1,0} \nonumber  \\
	   & = W_{0,j} x_{i,0} +\sum_{k=1}^{i-1} W_{k,j-1} B x_{i-k,0} \nonumber \\
	   & = \sum_{k=0}^{i-1} W_{k,j-1} B x_{i-k,0} 
\end{align}
Similarly,  the second sum 
\begin{align} \label{xij2}
	 \sum_{k=1}^j W_{i,j-k}(x_{0,k}-Bx_{0,k-1}  ) = \sum_{k=0}^{j-1} W_{i-1,k} A x_{0,j-k}
\end{align}
	Replacing the sums in \eqref{xij} by \eqref{xij1} and \eqref{xij2}, we finally obtain
\begin{equation*}
	x_{i,j} = \sum_{k=0}^{j-1} W_{i-1,k} A x_{0,j-k} + \sum_{k=0}^{i-1} W_{k,j-1} B x_{i-k,0} 
\end{equation*}	
as desired. Finally, if $A$ and $B$ commute, every word containing $i$ copies of $A$ and $j$ copies of $B$ equals $A^iB^j$. Since there are $\binom{i+j}{i}$ such words, we have
\[ W_{i,j}(A,B)=\binom{i+j}{i}A^iB^j. \]
This completes the proof.
\end{proof}

\begin{rem} \label{rem1}
	As a straightforward consequence of Proposition~\ref{BigForm}, together with the fact that the boundary data may be chosen independently on the two axes, the system admits the following characterization. In particular, when testing one of the two contributions in \eqref{BigFormEqn}, the boundary data on the other axis may be set identically equal to zero, so no cancellation between the two sums is involved.
	\begin{enumerate}
		\item $\mathcal{S}(A,B)$ is \emph{asymptotically stable} if and only if both sums on the right-hand side of equation \eqref{BigFormEqn} converge to zero as $i,j \to \infty$, for all admissible boundary data.
		
		\item $\mathcal{S}(A,B)$ is \emph{neutrally stable} if and only if both sums in equation~\eqref{BigFormEqn} remain bounded for every admissible boundary data, while for some admissible boundary data at least one of the two contributions does not tend to zero.
		
		\item $\mathcal{S}(A,B)$ is \emph{unstable} if and only if one of the two one-axis contributions in equation \eqref{BigFormEqn} is unbounded for some admissible boundary data. Equivalently, instability can be detected by choosing boundary data on one axis and taking the data on the other axis to be identically zero.
	\end{enumerate}
\end{rem}

\medskip

Since we do not require the boundary conditions to vanish at infinity and we adopt Pringsheim convergence, our definition of attractivity does not \emph{a priori} rule out the possibility that \(x_{i,j}\) becomes unbounded along a fixed vertical line \((i=\mathrm{constant})\) or a fixed horizontal line \((j=\mathrm{constant})\), while still satisfying \(x_{i,j}\to 0\) as \(i,j\to\infty\). In the \(c_0\)-boundary setting with convergence taken as \(i+j\to\infty\), attractivity is known to imply asymptotic stability for linear two-dimensional Roesser and Fornasini--Marchesini models; see \cite{Ba18}, and also \cite{BaClRiSiYe23} for related connections between stability notions. In Theorem \ref{theo:att Lyap} below, we prove the analogous implication in the bounded-boundary/Pringsheim framework used here. The proof below is independent of the arguments in these references and uses the explicit representation formula \eqref{BigFormEqn}.

We first prove two auxiliary lemmas, both of which are useful in their own right. The first lemma shows that attractivity of $S(A,B)$ forces attractivity of the one-dimensional systems generated by $A$ and by $B$. 

\begin{lem} \label{lem:Attractive}
	If the system ${\cal S}(A,B)$ is attractive, then so are the 1D systems generated by $A$ and $B$.
\end{lem}

\begin{proof}
	We prove the assertion for \(A\); the proof for \(B\) is analogous. Suppose, toward a  contradiction, that $A^iv \not\to 0$ for some $v \in \C^n$.  Then there exists $\varepsilon>0$ and a sequence $i_k \to \infty$ such that 
	\begin{equation*}
		\|A^{i_k}v\|\ge \varepsilon \qquad \text{for all } k .
	\end{equation*}
	Choose a strictly increasing sequence  \(  j_k \to \infty \). We prescribe boundary data on the \(j\)-axis only by setting  
	\[ x_{i,0}=0 \qquad (i\ge 1), \]
	and
	\begin{equation*}
		x_{0,j_k}=\lambda_k v,\qquad |\lambda_k|=1,
	\end{equation*}
	with all other \(x_{0,j}, \ (j\neq j_k)\) equal to zero. The constants \(\lambda_k\) will be chosen recursively. By \eqref{BigFormEqn}, at the point $(i_k,j_k)$, we have 
	\begin{align*}
		x_{i_k,j_k}
		&= \sum_{\ell=0}^{j_k-1} W_{i_k-1,\ell}(A,B) A x_{0,j_k-\ell} \\
		&= W_{i_k-1,0}(A,B)A x_{0,j_k}
		+ \sum_{\ell=1}^{j_k-1} W_{i_k-1,\ell}(A,B) A x_{0,j_k-\ell} \\
		&= \lambda_k A^{i_k}v + S_k,
	\end{align*}
	where 
	\[ S_k = \sum_{\ell=1}^{j_k-1} W_{i_k-1,\ell}(A,B) A x_{0,j_k-\ell} .  \]
	Since $j_k-\ell <j_k$ when $\ell \geq 1$, the vector $S_k$ depends only on  $\lambda_1, \ldots, \lambda_{k-1}$, and not on $\lambda_k$. Since all vector norms on  $\mathbb C^n$ are equivalent, we may use the norm induced by the standard inner product on $\C^n$ to obtain
	\[ \frac{1}{2\pi} \int_0^{2\pi} \| S_k + e^{it} A^{i_k} v\|^2 \, dt = \| S_k\|^2 + \|  A^{i_k} v \|^2. \]
	By the mean-value theorem for integrals, there exists $t_k$ such that 
	\[  \| S_k + e^{it_k} A^{i_k} v\|^2 = \|S_k\|^2 +  \| A^{i_k} v\|^2.  \]
	Choose $\lambda_k=e^{it_k}$. Then 
	\begin{equation*}
		\| x_{i_k,j_k} \| = \|S_k+\lambda_k A^{i_k}v\|\ge \|A^{i_k}v\|\ge \varepsilon .
	\end{equation*}
	This contradicts the attractivity of ${\cal S}(A,B)$. Therefore, the 1D system associated with $A$ is attractive,
\end{proof}

The second lemma gives the uniform tail estimate that will be needed in the proof of Theorem~\ref{theo:att Lyap}. For bounded boundary data
\[ \xi=\bigl((x_{r,0})_{r\geq1},(x_{0,s})_{s\geq1}\bigr), \]
set
\begin{equation} \label{norm bound}
	\|\xi\| = \max\left\{ \sup_{r\geq1}\|x_{r,0}\|, \sup_{s\geq1}\|x_{0,s}\| \right\}. 
\end{equation}	
For $i,j\geq1$, define
\begin{align}
	T_{i,j}\xi &= \sum_{k=0}^{j-1}W_{i-1,k}(A,B)A x_{0,j-k} + \sum_{k=0}^{i-1}W_{k,j-1}(A,B)B x_{i-k,0}. \label{Tij}
\end{align}
By Proposition~\ref{BigForm}, $T_{i,j}\xi=x_{i,j}$. Thus $T_{i,j}$ is the linear operator that assigns to the boundary data $\xi$ the value of the corresponding solution at $(i,j)$. The following lemma provides the uniform tail estimate needed below.

\begin{lem}\label{lem:uniform}
	If $\mathcal S(A,B)$ is attractive, then there exist $N\in\mathbb N$ and $M>0$ such that
	\[ \|T_{i,j}\xi\| \leq M\|\xi\| \qquad \text{for all }i,j\geq N \]
	and all bounded boundary data $\xi$.
\end{lem}

\begin{proof}
	Let $\mathcal B$ be the space of all bounded boundary data, endowed with the norm \eqref{norm bound}. This space is complete. For each fixed $i,j$, equation~\eqref{Tij} shows that $T_{i,j}$ is a continuous linear map from $\mathcal B$ into $\mathbb C^n$.
	
	Since $\mathcal S(A,B)$ is attractive, there exists $\delta>0$ such that
	\[ T_{i,j}\xi\longrightarrow0 \qquad\text{as }i,j\longrightarrow\infty \]
	whenever $\|\xi\|<\delta$. By linearity, this convergence holds for every $\xi\in\mathcal B$.
	
	For each $m\geq1$, define
	\[ F_m = \left\{ \xi \in \mathcal B: \|T_{i,j}\xi\|\leq1 \text{ for all }i,j\geq m \right\}.\]
	Each $F_m$ is closed, since
	\[ F_m = \bigcap_{i,j\geq m} \left\{ \xi \in \mathcal B:\|T_{i,j}\xi\|\leq1 \right\}.\]
	Moreover, since $T_{i,j}\xi\to0$ for every $\xi\in\mathcal B$,
	\[ \mathcal B=\bigcup_{m=1}^{\infty}F_m.\]
	
	By the Baire category theorem, some $F_N$ has nonempty interior. Hence there exist $\xi^0\in\mathcal B$ and $r>0$ such that
	\[ \{\xi\in\mathcal B:\|\xi-\xi^0\|<r\} \subseteq F_N. \]
	Thus, if $\|\eta\|<r$, then
	$\xi^0+\eta\in F_N$ and $\xi^0\in F_N$, so for all $i,j\geq N$,
	\[ \|T_{i,j}\eta\| \leq \|T_{i,j}(\xi^0+\eta)\| + \|T_{i,j}\xi^0\| \leq2.\]
	
	Now let $\|\xi\| \leq1$ and set
	\[ \eta=\frac r2\,\xi. \]
	Then $\|\eta\|<r$, and therefore
	\[ \frac r2\,\|T_{i,j}\xi\| \leq2 \qquad (i,j\geq N). \]
	Hence
	\[ \|T_{i,j}\xi\| \leq \frac4r \qquad (i,j\geq N). \]
	The result follows with $M=4/r$.
\end{proof}

We are now ready to relate attractivity to Lyapunov stability. The next theorem shows that, in the bounded-boundary/Pringsheim framework, attractivity alone implies asymptotic stability.

\begin{theo} \label{theo:att Lyap}
If $\mathcal{S}(A,B)$ is attractive, then it is Lyapunov stable. In particular,  $\mathcal{S}(A,B)$ is asymptotically stable if and only if it is attractive. 
\end{theo}

\begin{proof}
	By Lemma \ref{lem:uniform}, there exists $N \in \N$ for which 
	\[  \sup_{i \ge N, j \ge N } \| T_{i,j} \| < \infty. \]
	It remains to control the finitely many horizontal and vertical strips outside this tail region.  By Lemma  \ref{lem:Attractive} and the standard one-dimensional stability criterion, the systems
	generated by $A$ and $B$ are attractive; hence
	\[ \rho(A)<1,\qquad \rho(B)<1. \]
	
	Choose \(\alpha,\beta\in(0,1)\) and constants \(C_A,C_B>0\) such that
	\begin{equation*}
		\|A^m\|\le C_A\alpha^m, \qquad \|B^m\|\le C_B\beta^m \qquad (m\ge 0).
	\end{equation*}
    Fix $p \geq 0$. Every word containing $p$ copies of $A$ and $q$ copies of $B$ can be written in the form
    \[ B^{q_0} A B^{q_1}A \cdots A B^{q_p}, \qquad q_0+\cdots q_p=q. \]
    Thus each such word has  norm at most
    \[ \| B^{q_0} A B^{q_1}A \cdots A B^{q_p} \| \leq C_B^{p+1} \|A\|^p \beta^q. \]
    Since the number of such words is $\binom{p+q}{q}$, we get
    \[ \| W_{p,q}(A,B) \|  \leq C_B^{p+1}  \binom{p+q}{q} \|A\|^p \beta^q.  \]
	The negative binomial series then implies  
	\begin{align*}
		\sum_{q=0}^{\infty}\|W_{p,q}(A,B)\| & \leq C_B^{p+1} \|A\|^p \sum_{q=0}^\infty \binom{p+q}{q} \beta^q \\
		   & =  C_B^{p+1} \|A\|^p (1-\beta)^{-p-1} < \infty.
	\end{align*}
	Similarly, for each fixed \(q\ge 0\),
	\begin{equation*}
		\sum_{p=0}^{\infty}\|W_{p,q}(A,B)\| < \infty. 
	\end{equation*} 
	Now decompose \(T_{i,j}\) according to \eqref{BigFormEqn} as the sum of the two one-axis contributions
	\begin{equation*}
		T_{i,j}=U_{i,j}+V_{i,j},
	\end{equation*}
	where
	\begin{equation*}
		U_{i,j}\bigl((x_{0,r})_{r\ge 1}\bigr) = \sum_{r=0}^{j-1} W_{i-1,r}(A,B)A x_{0,j-r},
	\end{equation*}
	and
	\begin{equation*}
		V_{i,j}\bigl((x_{r,0})_{r\ge 1}\bigr) = \sum_{r=0}^{i-1} W_{r,j-1}(A,B)B x_{i-r,0}.
	\end{equation*}
	
	By Lemma \ref{lem:uniform}, the operators \(T_{i,j}\) are uniformly bounded when \(i\ge N\) and \(j\ge N\). Suppose first $1 \leq i < N$. Since $i$ belongs to a finite set, it suffices to prove uniform boundedness in $j$ for each fixed $i$: For this contribution, we have 
	\begin{align*}
		\|U_{i,j}\|  & \leq \|A\| \sum_{k=0}^{j-1} \| W_{i-1,k}(A,B) \| \\
		       & \leq \|A\| \sum_{k=0}^\infty \| W_{i-1,k}(A,B) \| < \infty
	\end{align*}
	Hence $\sup_{j\ge 1}\|U_{i,j}\|<\infty$. For the second contribution
	\[ \|V_{i,j}\| \leq \|B\| \sum_{k=0}^{i-1} \| W_{k,j-1}(A,B) \| .\]
	For each fixed $k$, the convergence of $ \sum_{q=0}^\infty \|W_{k,q}(A,B)\|$ implies that 
	\[ \sup_{q \ge 0} \| W_{k,q}(A,B) \| < \infty.\]
	Since the sum over $k$ is finite, it follows that 
	\[ \sup_{j \geq 1 } \| V_{i,j} \| < \infty. \]
	Taking the maximum over the finitely many indices $1 \leq i <N$, we obtain
		\begin{equation*}
		\sup_{1\le i<N,\ j\ge 1}\|T_{i,j}\|<\infty .
	\end{equation*}
	Similarly, if \(1\le j<N\), then the same estimates, with the roles of \(A\) and \(B\) interchanged, give
	\begin{equation*}
		\sup_{i\ge 1,\ 1\le j<N}\|T_{i,j}\|<\infty .
	\end{equation*}
	
	Combining these two strip bounds with the tail provided by Lemma~\ref{lem:uniform}, we obtain
	\begin{equation*}
		M:=\sup_{i,j\ge 1}\|T_{i,j}\|<\infty .
	\end{equation*}
	Therefore, for every boundary data \(\xi\in \mathcal B\),
	\begin{equation*}
		\sup_{i,j\ge 1}\|x_{i,j}\| = \sup_{i,j\ge 1}\|T_{i,j}\xi\| \le M\|\xi\| .
	\end{equation*}
	Given \(\varepsilon>0\), choose \(\delta=\min \left\{ \varepsilon, \frac{\varepsilon}{M+1}\right\}\). Then $\|\xi\|<\delta $ implies that the interior values satisfy 
	\begin{equation*}
         \|x_{i,j}\|<\varepsilon \qquad (i,j\ge 1),
	\end{equation*}
	and the boundary values also satisfy $\|x_{i,0}\| < \varepsilon$ and $\|x_{0,j} \| < \varepsilon$. Hence \(\mathcal S(A,B)\) is Lyapunov stable.
\end{proof}

Our next result establishes a recurrence relation satisfied by the sum of the anti-diagonals of the sequence $\{x_{i,j}\}$, defined by  
\begin{equation} 
y_n = \sum_{k=0}^n x_{k,n-k}.
\end{equation}
To the best of our knowledge, this relation has not appeared in the literature before. We use this relation to show that $\mathcal{S}(A, B)$ is unstable if $A+B$ has an eigenvalue outside the unit disk. 

\begin{Prop} \label{anti}
	Let $y_n$ be as above. We have 
	\begin{align}  	\label{bigrec}
		y_{n+1} &= (A+B) y_n + f_n = \sum_{k=0}^n (A+B)^{n-k} f_k 
    \end{align}
	where
	\[  f_n = x_{n+1,0} + x_{0,n+1} - A x_{n,0} - B x_{0,n} \]
	depends only on the boundary conditions.
\end{Prop}

\begin{proof}
	The recurrence relation \eqref{S(A,B)} for $x_{i,j}$ implies 
	\begin{align*}
		y_{n+1} & = \sum_{k=0}^{n+1} x_{k,n-k+1} \\
		& =  \sum_{k=1}^{n} Ax_{k-1,n-k+1} + \sum_{k=1}^{n} B x_{k,n-k} +x_{n+1,0} + x_{0,n+1}\\
		& = \sum_{k=0}^{n} (Ax_{k,n-k} + B x_{k,n-k}) + (x_{0,n+1} +x_{n+1,0} - A x_{n,0} -Bx_{0,n}) \\
		& = (A+B)y_n + f_n
	\end{align*}
	as desired. Since $y_0=x_{0,0}=0$, the recurrence gives $y_1=f_0$. Iterating
	\begin{align*}
		y_{n+1} & = (A+B) y_n +f_n 
\end{align*}
then yields 
\begin{align*}	
	y_{n+1}	& = (A+B)\sum_{k=0}^{n-1} (A+B)^{n-1-k} f_k + f_n \\
		& = \sum_{k=0}^{n} (A+B)^{n-k} f_k,
	\end{align*}
	so the second equality in \eqref{bigrec} follows by induction. 
\end{proof}

\begin{cor} \label{unstablegeneral}
	If $\rho(A+B) > 1$, then $\mathcal{S}(A,B)$ is unstable.
\end{cor}

\begin{proof}
	Consider the 1D system associated with $A+B$. By assumption,  $A+B$ has an eigenvalue $|\lambda| >1$ with eigenvector $v$. By Proposition~\ref{anti}, with boundary conditions 
	\[ x_{1,0} = x_{0,1} = v, \quad x_{i,0} = x_{0,j} = 0 \ (i,j > 1), \] 
	we have
	\[ f_0 = 2v, \ f_1 = -(A+B)v = -\lambda v, \quad f_j = 0 \ ( j>1).	\]
	Thus, it follows that 
	\begin{align*}
		y_{n} & = \sum_{k=0}^{n} x_{k,n-k}
		\nonumber
		\\
		& = (A+B)^{n-1}f_0 + (A+B)^{n-2} f_1
		\nonumber
		\\
		&= \lambda^{n-1} v
	\end{align*}
	Therefore
	\begin{equation} \label{eqn22}
		\left\| \sum_{k=0}^{n} x_{k,n-k} \right\| = |\lambda|^{n-1}\|v\| 
	\end{equation}  
	If there exists $M>0$ such that $\| x_{k,n-k} \| < M$ for all $k$ and $n$, then 
	\[ \left\| \sum_{k=0}^n x_{k,n-k} \right\| \leq \sum_{k=0}^n \| x_{k,n-k} \| \leq (n+1)M \]
	contradicting \eqref{eqn22}. Hence, there must be a subsequence $\{ x_{i_k,j_k} \}$ for which
	\[  \| x_{i_k,j_k} \| \to \infty \qquad \text{ as } k \to \infty. \]
	Thus the solution is unbounded, and hence $S(A,B)$ is not Lyapunov stable. 
\end{proof}

\section{The commutative case}

We begin by analyzing the stability of the system \(\mathcal{S}(A,B)\) when the matrices \(A\) and \(B\) commute. In this case, the stability of ${\cal S}(A,B)$ can be described explicitly in terms of paired spectral data for \(A\) and \(B\).

This section is related to the work of Bachelier et al.~\cite{BaCl23}, where asymptotic stability is characterized for commuting Fornasini--Marchesini models under the additional assumption that one of the two matrices is diagonalizable. Our setting and conclusions are different in several respects: we do not impose any diagonalizability assumption, we work with bounded boundary data and Pringsheim convergence rather than \(c_0\)-boundary data and convergence as \(i+j\to\infty\), and we classify not only asymptotic stability but also neutral stability and instability. The proofs below are also adapted to this functional setting and use the explicit representation formula to read off the growth or decay of the relevant coefficient matrices.

\subsection*{The scalar case}

It is useful to first consider the situation where $A=a$ and $B=b$ are scalar. More precisely, we consider the system
\begin{equation} \label{2D Scalar} 
\begin{cases} 
       x_{i,j} = a x_{i-1,j}+bx_{i,j-1} \\
       \{ x_{i,0} \}_{i=1}^\infty \in \ell^\infty, \quad  \{ x_{0,j}\}_{j=1}^\infty \in \ell^{\infty}.
  \end{cases}
\end{equation}  
In \cite{Ba19}, it is shown that $x=0$ is asymptotically stable as $i+j \to \infty$, with boundary data in $c_0$, if and only if one of the following conditions holds:
\begin{itemize}
\item $|a|+|b|\leq 1 \text{ if } a\neq 0 \text{ and } b\neq 0$,
\item $|a|<1 \text{ if } b=0$,
\item $|b|<1 \text{ if } a=0$.
\end{itemize}
However, it follows directly from the recurrence relation \eqref{2D Scalar} that $x=0$ remains attractive when $|a|=1$, $b=0$, or when $|b|=1$, $a=0$, under their assumptions.
Consequently, the correct formulation of their result should be that $x=0$ is asymptotically stable if and only if $|a|+|b|\leq 1$.

The proof given in \cite{Ba19} is computationally intensive; it involves a detailed case-by-case analysis and relies on  Stirling's formula to estimate behavior of terms of the form $\binom{r+s}{r} a^r b^s$. In Theorem \ref{Scalar} below, we extend this result to the more general setting described in \eqref{2D Scalar} using a much simpler argument based on the negative binomial series.  The theorem provides a more refined analysis by distinguishing between cases in which the system is asymptotically stable and those in which it is neutrally stable. 
 
\begin{theo} \label{Scalar}
The  system $\mathcal{S}(a,b)$ is:
\begin{enumerate}
  \item[(a)] asymptotically stable if $|a|+|b|<1$,
  \item[(b)] neutrally stable if $|a|+|b|=1$, 
  \item[(c)] unstable if $|a|+|b|>1$.
\end{enumerate}
\end{theo}

\begin{proof}
As a consequence of Proposition \ref{BigForm}, the solution $x_{i,j}$ admits the representation
\begin{equation} \label{Scalar eq0}  
x_{i,j} = a^i \sum_{k=0}^{j-1} \binom{i-1+k}{k} b^k x_{0,j-k} + b^j \sum_{k=0}^{i-1} \binom{j-1+k}{k} a^k x_{i-k,0}.
\end{equation}
By Remark \ref{rem2}, we may assume without loss of generality that $a$ and $b$ are non-negative real numbers.

\smallskip

\noindent \textit{(a)} Suppose $\|\{x_{i,0}\}\|_\infty, \|\{x_{0,j}\}\|_\infty \leq \delta$ for some $\delta>0$. Then, by the negative binomial series, it follows for $a,b \in [0,1)$ that 
\begin{equation} \label{Scalar eq1}
| x_{i,j} | \leq  \left[  \left( \frac{a}{1-b}  \right)^i   +  \left( \frac{b}{1-a} \right)^j   \right] \delta.
\end{equation}
Therefore, the fixed point $x=0$ is  attractive if 
\[  \frac{a}{1-b} <1 \qquad \text{and} \qquad \frac{b}{1-a}  <1, \]
which is equivalent to the condition $a+b<1$. Consequently, $\mathcal{S}(a,b)$ is asymptotically stable when $a+b<1$.

\medskip

\noindent \textit{(b)} Now, assume that $a+b=1$. We consider several distinct cases. First, suppose $a=1$, $b=0$. In this case, the recurrence relation simplifies to
\[ x_{i,j}= x_{0,j} \quad \forall i,j \in \N. \]
This clearly implies that $\mathcal{S}(a,b)$  is Lyapunov stable, but not attractive. A similar argument applies when $a=0$ and $b=1$.

Next, consider the case $a+b=1$ with $a,b \in (0,1)$. From equation \eqref{Scalar eq1}, we have
\[ |x_{i,j}| = \Oo(\delta) \quad \forall \, i,j \in \N, \]
which implies Lyapunov stability of the fixed point $x=0$. To show that $\mathcal{S}(a,b)$ is not attractive, we first consider the specific case $a=b=\frac{1}{2}$, and choose the boundary conditions
\[  x_{k,0} = x_{0,k} =  \delta>0 \quad \forall k \in \N.\]
Using equation \eqref{Scalar eq0}, the diagonal orbit $x_{i,i}$ becomes
\begin{align*}
  x_{i,i} & = 2^{-i} \sum_{k=0}^{i-1} \binom{i-1+k}{k} 2^{-k} \delta +  2^{-i} \sum_{k=0}^{i-1} \binom{i-1+k}{k} 2^{-k} \delta \\
     & = 2^{-i} 2^{i-1} \delta + 2^{-i} 2^{i-1} \delta \\
     & = \delta.
\end{align*}
Thus, the orbit remains constant along the diagonal, and the fixed point $x=0$ is not attractive in this case. 

Finally, consider the case $a+b=1$ with $0<a<1/2$, and choose boundary conditions 
\begin{equation} \label{IC 2}
 x_{k,0} =\delta>0 \qquad \text{and} \qquad x_{0,k} = 0 \qquad \forall k \in \N
\end{equation}
Then, it follows from the inequality
\begin{equation} \label{Scalar eq2}
\sum_{k=0}^{i-1} \binom{i-1+k}{k} \, t^k \geq \frac{1}{(1-t)^{i-1}} \qquad ( t \in (0,1/2))
\end{equation} 
that 
\begin{align*}
  x_{i,i} & = b^i \sum_{k=0}^{i-1} \binom{i-1+k}{k} a^k \delta  \\
   & \geq \frac{b^i}{(1-a)^{i-1}} \delta \\
   & = b \delta
\end{align*}
which implies again that $\mathcal{S}(a,b)$ is not attractive. A similar argument applies in the case $0<b<1/2$. This completes the proof of \textit{(b)}.

\medskip

\noindent \textit{(c)} To conclude, we note that Corollary~\ref{unstablegeneral} implies that $\mathcal{S}(a,b)$ is unstable when $a+b>1$.
\end{proof}

\begin{rem} \label{rem4}
	In the scalar case, the stability criterion depends on the functional setting in which the boundary-value problem is posed. For instance, under bounded boundary data and Pringsheim convergence, the corresponding asymptotic stability condition is $|a|+|b|<1$. This differs from the $c_0$-boundary setting considered in, for example, \cite{Ba19,BaCl23}, where the borderline case $|a|+|b|=1$ may still be admissible. Thus, even in the scalar case, the stability property is determined by the system only after the class of admissible boundary data and the convergence regime have been fixed.
\end{rem}

\subsection*{The matrix case}

We now turn to the matrix case when $A$ and $B$ are commuting matrices. We begin with the generic situation in which both $A$ and $B$ are diagonalizable. It is well-known that commuting diagonalizable matrices can be simultaneously diagonalized; that is, there exists an invertible matrix $X$ such that   

\begin{equation} \label{Order}
X^{-1}AX = \begin{bmatrix} a_1  \\  & \ddots \\ & &  a_n\end{bmatrix} \quad \mbox{and} \quad X^{-1}BX = \begin{bmatrix} b_1 \\  & \ddots \\ & & b_n\end{bmatrix}
\end{equation} 
where $a_1,\ldots,a_n$ and $b_1,\ldots,b_n$ are the eigenvalues of $A$ and $B$, respectively.  

It is straightforward to observe that, under these assumptions, the stability analysis of the system reduces to verifying the scalar stability conditions for each pair of eigenvalues $(a_k,b_k)$. The following result is thus an immediate consequence of Theorem \ref{Scalar}.

\begin{cor} \label{diagonal}
Let $A$, $B \in \mathcal{M}_n(\C)$ be commuting, diagonalizable matrices satisfying \eqref{Order}.  Then, the system $\mathcal{S}(A,B)$ is:
\begin{enumerate}
  \item[(a)] asymptotically stable if $|a_k| + |b_k| < 1 $ for each $k=1,\dots, n$, 
  \item[(b)] neutrally stable if $|a_k| + |b_k| \leq 1$ for each $k=1,\dots,n$, with equality for at least one $k$,
  \item[(c)] unstable if $|a_k| + |b_k| > 1$ for some $k$.
\end{enumerate}
\end{cor}

\begin{ex}
Consider the system \eqref{S(A,B)} with $A$ normal and $B=A^*$ (its conjugate transpose). Then, by Corollary \ref{diagonal}, the fixed point $x=0$ is asymptotically stable if and only if the spectral radius $\rho(A)$ satisfies $\rho(A)<1/2$. The system is neutrally stable if $\rho(A)=1/2$, and unstable if $\rho(A)>1/2$. 
\end{ex}

We now turn to the general commuting case, in which one or both of the matrices $A$ and $B$ are defective. In this setting, additional care is required to analyze the system's stability. In particular, one must first establish a correspondence between the eigenvalues of $A$ and $B$. This is achieved through the following factorization result due to O'Meara and Vinsonhaler \cite{OmVi06}. 

\begin{lem} \label{OV}
Let $A,B \in \mathcal{M}_n(\C)$ be two commuting matrices. Then, there exists an invertible matrix $X \in \mathcal{M}_n(\C)$ such that both $X^{-1}AX$ and $X^{-1}BX$ are block diagonal matrices with the same block structure:
\begin{equation} \label{Factorization}
X^{-1}AX = \begin{bmatrix}
   A_1 \\
   & \ddots \\
   & &  A_m
\end{bmatrix} 
\quad \mbox{and} \quad
X^{-1}BX = \begin{bmatrix}
   B_1 \\
   &  \ddots \\
   & &  B_m
\end{bmatrix} 
\end{equation}
where each pair of corresponding blocks $A_k$ and $B_k$ commutes. Moreover, each of the blocks $A_k$ and $B_k$ has a single eigenvalue. 
\end{lem}

We now remove the diagonalizability assumption. The simultaneous block decomposition in Lemma~\ref{OV} allows us to pair the spectral data of \(A\) and \(B\) block by block, leading to the following classification.

\begin{theo} \label{commutative}
Let $A,B \in \mathcal{M}_n(\C)$ be two commuting matrices. Let $a_k$ and $b_k$ denote the (single) eigenvalues associated with the corresponding blocks $A_k$ and $B_k$ in the block-diagonal factorization \eqref{Factorization} of $A$ and $B$ given by Lemma \ref{OV}. Then, the system $\mathcal{S}(A,B)$ is:
\begin{enumerate}
  \item[(a)] asymptotically stable if $|a_k| + |b_k| < 1$ for each $k=1,\dots,m$,  
  \item[(b)] neutrally stable if $|a_k|+|b_k| \leq 1$ for each $k=1,\dots,m$, with equality for some $k$, provided that the eigenvalues $a_k, b_k$ with $|a_k|+|b_k| = 1$ are semisimple, 
  \item[(c)] unstable otherwise. 
\end{enumerate}
\end{theo}

\begin{proof}
As a consequence of Lemma \ref{OV}, it suffices to study the stability of the systems individually for each pair of blocks $A_k$ and $B_k$. By Remark~\ref{rem2}, we may, without loss of generality, assume for each fixed block that the single eigenvalues $a_k$ and $b_k$ associated with $A_k$ and $B_k$, respectively, are non-negative; that is $a_k,b_k \geq 0$.  We denote by $x_{i,j}^{(k)}$ the vector corresponding to the system $\mathcal{S}(A_k,B_k)$. 

\smallskip

\noindent \textit{(a)} Let $\delta>0$, and assume that the boundary conditions satisfy 
\[  \| x_{i,0}^{(k)}\|_\infty <\delta, \ \|x_{0,j}^{(k)}\|_\infty <\delta \]
for all $i,j \in \N$. Let $\varepsilon>0$ be chosen sufficiently small such that 
\[   a_k+b_k+2\varepsilon<1.\]
By Lemma  5.6.10 on page 297 of \cite{HoJo85}, there exist constants $C_1=C_1(A_k,\varepsilon)>0$ and $C_2=C_2(B_k,\varepsilon)>0$ such that 
\begin{equation*}
\|A_k^r\| \leq C_1 ( a_k+\varepsilon)^r \quad \text{ and } \quad \| B_k^r \| \leq C_2 ( b_k+\varepsilon)^r
\end{equation*}
for all $r \in \N$. Applying Proposition \ref{BigForm} together with these bounds, we obtain the following estimate
\begin{align} \label{positive}
    \|x_{i,j}^{(k)} \|_\infty & \leq C_1 C_2 \delta (a_k+\varepsilon)^i  \sum_{l=0}^{j-1} \binom{i-1+l}{l} (b_k+\varepsilon)^l \\
     & \ \ + C_1 C_2 \delta (b_k+\varepsilon)^j \sum_{l=0}^{i-1} \binom{j-1+l}{l} (a_k+\varepsilon)^l.
\end{align}
Then, by the negative binomial series, it follows that
\begin{equation*}
  \| x_{i,j}^{(k)} \|_\infty = \Oo \left( \, \left( \frac{a_k+ \varepsilon}{1-(b_k+\varepsilon)} \right)^i  \right)+ \Oo \left( \, \left( \frac{b_k+ \varepsilon}{1-(a_k+\varepsilon)} \right)^j \, \right)
\end{equation*}
which decays to zero as $i,j \to \infty$. This implies that ${\cal S}(A,B)$ is asymptotically stable. 

\medskip 

\noindent \textit{(b)} Now, suppose that $a_k+b_k \leq  1$ for all $k \in \{1,\ldots,m\}$, with equality holding for at least one index $k$. By part \textit{(a)}, it suffices to focus on the blocks $A_k$ and $B_k$ for which $a_k+b_k=1$. By assumption, the eigenvalues $a_k$ and $b_k$ are semisimple; that is, their geometric multiplicities equal their algebraic multiplicities. Consequently, the blocks $A_{k}$ and $B_{k}$ can be simultaneously diagonalized. Since each block has only one eigenvalue, semisimplicity implies \(A_k=a_k I\) and \(B_k=b_k I\) on that block. Hence the conclusion follows directly from part \textit{(b)} of Corollary \ref{diagonal}. 

\medskip 

\noindent \textit{(c)} First, suppose that $a_k+b_k>1$ for some $k \in \{1,\ldots,m\}$. Since $A_k$ and $B_k$ commute, they must share a common eigenvector, which we denote by $v_k$. Then $(A_k+B_k)v_k = (a_k+b_k)v_k$, so $\rho(A_k+B_k)>1$. The system $\mathcal{S}(A_k,B_k)$ is therefore unstable. The block decomposition ~\eqref{Factorization} shows that $\mathcal{S}(A_k,B_k)$ is an invariant subsystem of  $\mathcal{S}(A,B)$. Choosing boundary data supported entirely in the $k$th block therefore produces an unbounded solution of the full system. Hence $\mathcal{S} (A,B)$ is unstable.  

Next, we consider the situation where $a_k+b_k=1$, but either $a_k$ or $b_k$ is not semisimple. By Theorem 2.3.3 in \cite{HoJo85}, we may assume without loss of generality that $A_k$ and $B_k$ are simultaneously upper triangular.  Hence, in a suitable basis of the $k$-th block, we may write
\[ A_k=a_kI+N_k,\qquad B_k=b_kI+M_k, \]
where $N_k$ and $M_k$ are strictly upper triangular nilpotent commuting matrices. Since either $A_k$ or $B_k$ is not semisimple, at least one of $N_k$ and $M_k$ is nonzero. Let $r$ be the first column for which either $N_k$ or $M_k$ has a nonzero entry. Then $N_k$ and $M_k$ vanish on
\[ W=\operatorname{span} \{e_1,\ldots,e_{r-1} \}. \]
Choose $s<r$ such that the $e_s$-coefficient of either $N_ke_r$ or $M_ke_r$ is nonzero, and set
\[ L=\operatorname{span} \{e_1,\ldots,e_{s-1},e_{s+1},\ldots,e_{r-1}\}. \]
Since $N_k$ and $M_k$ vanish on $W$, the subspace $L\subset W$ is invariant under both $N_k$ and $M_k$. Let
\[ U=W+\operatorname{span} \{e_r\}. \]
Then $U$ is also invariant under $N_k$ and $M_k$, and therefore the quotient space $U/L$ is invariant under the induced maps. The quotient $U/L$ is two-dimensional, with basis given by the classes of $e_s$ and $e_r$. In this basis, the induced nilpotent maps have the form
\[ \widetilde N_k=
\begin{bmatrix}
	0&c_1\\
	0&0
\end{bmatrix},
\qquad
\widetilde M_k=
\begin{bmatrix}
	0&c_2\\
	0&0
\end{bmatrix}, \]
where $(c_1,c_2)\neq(0,0)$. Consequently, the induced pair is
\[ \widetilde A_k=
\begin{bmatrix}
	a_k&c_1\\
	0&a_k
\end{bmatrix},
\qquad
\widetilde B_k=
\begin{bmatrix}
	b_k&c_2\\
	0&b_k
\end{bmatrix}. \]
Moreover, if the original block system ${\cal S}(A_k,B_k)$ were Lyapunov stable, then every quotient system obtained from common invariant subspaces would also be Lyapunov stable. Therefore, it is enough to prove that the quotient system ${\cal S}(\widetilde A_k,\widetilde B_k)$ is not Lyapunov stable.

We first assume that \(a_k > 0\) and \(b_k > 0\).  For every $i,j \in \N$, we have 
	\[ \widetilde A_k^i \widetilde B_k^j e_2 = a_k^ib_k^j e_2+ a_k^i b_k^j \left( \frac{i c_1}{a_k}+\frac{j c_2}{b_k} \right)e_1, \]
and therefore
	\begin{align*} 
	 \sum_{l=0}^{i-1} W_{l,j-1}(\widetilde A_k, \widetilde B_k) \widetilde B_k e_2  
	 & =  \sum_{l=0}^{i-1} \binom{j-1+l}{l}  \widetilde A_k^l \widetilde B_k^j e_2 \\
	 & =  \sum_{l=0}^{i-1} \binom{j-1+l}{l} \left( a_k^l b_k^j e_2+ a_k^l b_k^j \left( \frac{l c_1}{a_k}+\frac{j c_2}{b_k} \right)  e_1 \right)
	\end{align*}
	For each fixed \(j\), set \(x_{0,s}=0\) for all \(s\geq1\), and choose the lower boundary values \(x_{i-l,0}\), \(0\leq l\leq i-1\), to be unimodular multiples of \(e_2\) so that all nonzero \(e_1\)-components in the sum have the same complex direction. Then, the $e_1$-direction of $x_{i,j}$ satisfies
	\[ \bigl|[x_{i,j}]_{e_1}\bigr|  =  \sum_{l=0}^{i-1}  \binom{j-1+l}{l} a_k^l b_k^j \left| \frac{l c_1}{a_k}+\frac{j c_2}{b_k} \right| \]
	Letting \(i\) be large enough, it is therefore enough to show that
	\[ S_j:= \sum_{l=0}^{\infty} \binom{j-1+l}{l} a_k^l b_k^{j} \left| \frac{l c_1}{a_k}+\frac{ j c_2}{b_k} \right| \]
	is unbounded as \( j \to\infty\). Indeed, for each fixed \(j\), the finite sums above increase to \(S_j\); hence we may choose \(i=i_j\) large enough so that the finite sum is at least
	\(S_j/2\). The boundary data may depend on \(j\), since Lyapunov stability would require a uniform bound over all admissible bounded boundary data.
	
	Now let $X_j$ be a random variable with negative binomial probability distribution with parameters $j$ and $b_k$.  The expected value and variance of $X_j$ are
	\[ \mathbb E [ X_j] =\frac{ja_k}{b_k}, \qquad \operatorname{Var}[X_j]=\frac{j a_k}{b_k^2}. \]
	Hence, $S_j$ can be expressed as 
	\[ S_j = \mathbb E \left[  \left| \frac{c_1}{a_k}X_j+\frac{jc_2}{b_k} 	\right| \right]  \]
	If $c_1 + c_2 \neq 0$, then 
	\[ S_j \geq \left| \mathbb E \left[  \frac{c_1}{a_k}X_j+\frac{jc_2}{b_k} \right] \right| = \frac{j}{b_k} |c_1 + c_2|  \]
	so \(S_j \to\infty\).
	
	On the other hand, if $c_1 + c_2 =0$, then necessarily \(c_1\neq0\), because \((c_1,c_2)\neq(0,0)\). In this case
	\[ S_j = \frac{|c_1|}{a_k} \mathbb E \left[ \, \left|X_j-\mathbb E [X_j ] \right| \, \right].\]
	The random variable  \(X_j\) is a sum of \(j\) independent geometric random variables with fixed success probability \(b_k \), so
	\(\operatorname{Var}[X_j] \sim j\) and its fourth central moment is \(O(j^2)\). By Hölder's inequality,
	\[ \mathbb E [ \, |W| \,] \ge \frac{(\mathbb E [ \, W^2 \, ] )^{3/2}}{(\mathbb E [ \,  W^4 \, ] )^{1/2}}, \qquad W=X_j-\mathbb E [ \,  X_j \, ], \]
	which gives
	\[ \mathbb E [ \, |X_j-\mathbb E [ \, X_j \, ]  \, | \, ] \ge C\sqrt{j}. \]
	Thus \(S_j \to\infty\) in this case as well.
	
	It remains to consider the degenerate cases. Suppose first that \(a_k=0\) -- the case where $b_k=0$ is similar. Then \(b_k=1\). If \(c_2\neq0\), then \( \widetilde B_k\) is a nontrivial Jordan block with unimodular eigenvalue. Taking \(x_{1,0}=e_2\), \(x_{r,0}=0\) for \(r\neq1\), and \(x_{0,s}=0\), Proposition~\ref{BigForm} gives
	\[ x_{1,j}=\widetilde B_k^j e_2 = b_k^j e_2+j b_k^{j-1}c_2 e_1, \]
	which is unbounded. If \(c_2=0\), then \(c_1\neq0\). Taking \(x_{1,0}=e_2\), \(x_{r,0}=0\) for \(r\neq1\), and \(x_{0,s}=0\), we obtain, for \(i=2\),
	\[ x_{2,j} = W_{1,j-1}(\widetilde A_k, \widetilde B_k) \widetilde B_k e_2 = j b_k^j c_1 e_1,\]
	which is unbounded. 	
	
	Thus, in every case with \(|a_k|+|b_k|=1\) and \((c_1,c_2)\neq(0,0)\), bounded boundary data can produce an unbounded solution. Hence \(\mathcal S(A,B)\) is not Lyapunov stable.
\end{proof}

As a simple corollary of the preceding results, we provide some useful sufficient conditions for the stability of the fixed point $x=0$ in terms of the spectral radii of $A$ and $B$. 

\begin{cor} \label{spectral radius cond}
Let $A$, $B \in \mathcal{M}_n(\C)$ be two commuting matrices.  If $\rho(A)+\rho(B)<1$, then $\mathcal{S}(A,B)$ is asymptotically stable. 
\end{cor}

\begin{proof}
	Let $a_k$ and $b_k$ be as in Theorem \ref{commutative}. Since
	\[  |a_k| + |b_k| \leq \rho(A) + \rho(B) < 1\]
	for all $k$, the conclusion follows from part \emph{(a)} of the same theorem. 
\end{proof}

\begin{ex} \label{ex1}
The previous corollary does not hold if $A$ and $B$ do not commute. To see this,  consider the nilpotent matrices
\[ A = \begin{bmatrix} 0 & 1 \\ 0 & 0 \end{bmatrix} \quad \mbox{ and } \quad 
B = \begin{bmatrix} 0 & 0 \\ 2 & 0 \end{bmatrix}
\]
Then $\rho(A)=\rho(B)=0$, but $\rho(A+B) = \sqrt{2}$,  so, by Corollary~\ref{unstablegeneral}, ${\cal S}(A,B)$ is unstable.
\end{ex} 

\begin{ex}
If $A$ and $B$ commute, $A$ is nilpotent and $\rho(B)<1$, then ${\cal S}(A,B)$ is asymptotically stable since $\rho(A)=0$. As the previous example shows, this is not true if $A$ and $B$ do not commute. 
\end{ex}

As the following examples illustrate, outside the sufficient condition in Corollary \ref{spectral radius cond}, no definitive conclusion regarding the stability of ${\cal S}(A,B)$ can be made solely in terms of the spectral radii of $A$ and $B$. 

\begin{ex} \label{ex2}
When $\rho(A)+\rho(B) = 1$, the fixed point $x=0$ may be asymptotically stable, neutrally stable or unstable, depending on the specific structure of the matrices. For instance, by Theorem \ref{commutative}:
\begin{itemize}
\item ${\cal S}(A,B)$ is asymptotically stable for 
\[  A = \begin{bmatrix} 0.2 & 0 \\ 0 & 0.5 \end{bmatrix} \quad \mbox{and} \quad  B = \begin{bmatrix} 0.5 & 0 \\ 0 & 0.4\end{bmatrix}, \]
\item ${\cal S}(A,B)$ is neutrally stable for 
\[  A = \begin{bmatrix} 0.2 & 0 \\ 0 & 0.5 \end{bmatrix} \quad \mbox{and} \quad  B = \begin{bmatrix} 0.5 & 0 \\ 0 & 0.5\end{bmatrix}.\]
\item ${\cal S}(A,B)$ is unstable for 
\[  A = \begin{bmatrix} 0.5 & 0 \\ 0 & 0.5 \end{bmatrix} \quad \mbox{and} \quad  B = \begin{bmatrix} 0.5 & 1 \\ 0 & 0.5\end{bmatrix}.\]
\end{itemize}
\end{ex}

\begin{ex} \label{ex3} 
Similarly, no definitive conclusion can be drawn regarding the stability of the fixed point $x=0$ when $\rho(A)+\rho(B)>1$, even if $\rho(A),\rho(B)<1$. In this regime, ${\cal S}(A,B)$ may  be asymptotically stable, neutrally stable or unstable, as illustrated by the following three examples:
\begin{itemize}
\item ${\cal S}(A,B)$ is asymptotically stable for
\[  A = \begin{bmatrix} 0.2 & 0 \\ 0 & 0.5 \end{bmatrix} \quad \mbox{and} \quad  B = \begin{bmatrix} 0.6 & 0 \\ 0 & 0.4\end{bmatrix}, \]
\item ${\cal S}(A,B)$ is neutrally stable for 
\[  A = \begin{bmatrix} 0.2 & 0 \\ 0 & 0.5 \end{bmatrix} \quad \mbox{and} \quad  B = \begin{bmatrix} 0.6 & 0 \\ 0 & 0.5\end{bmatrix}, \]
\item ${\cal S}(A,B)$ is unstable for
\[  A = \begin{bmatrix} 0.2 & 0 \\ 0 & 0.5 \end{bmatrix} \quad \mbox{and} \quad  B = \begin{bmatrix} 0.5 & 0 \\ 0 & 0.6\end{bmatrix}.\]
\end{itemize}
\end{ex}

\section{The general matrix case}

In this section, we derive several general results concerning the behavior of the two-dimensional system \(\mathcal{S}(A,B)\), where \(A\) and \(B\) are arbitrary complex matrices. We begin with a fundamental result giving several equivalent characterizations of asymptotic stability. Some of these equivalences are classical. In particular, the equivalence between asymptotic stability and the determinant condition in part~\((d)\) is closely related to the internal-stability criterion of Fornasini and Marchesini \cite{FoMa78,FoMa80}, although their formulation uses boundary conditions prescribed along diagonals rather than along the coordinate axes used here. The spectral-radius condition in part~\((b)\) also appears in the literature; see, for instance, \cite{EbItHa06,FoMa78}. The purpose of the theorem below is to record these criteria in the bounded-boundary/Pringsheim framework of the present paper and to include the torus condition in part~\((c)\), which, to the best of our knowledge, is new.

\begin{theo}[Fornasini--Marchesini]  \label{bigTH}
Let $A,B \in \mathcal{M}_n(\C)$. The following conditions are equivalent:
\begin{enumerate}
 \item[(a)] The system $\mathcal{S}(A,B)$ is asymptotically stable;
 \item[(b)] $\rho(A+e^{i\omega}B)<1$ for every $\omega \in \R$;
 \item[(c)] $\rho(z_1 A+z_2 B)<1$ for every $(z_1,z_2) \in  \overline{\mathbb D}^2$;
 \item[(d)] $\det(I-z_1 A - z_2 B ) \neq 0$ for every $(z_1,z_2) \in \overline{\mathbb{D}}^2$.
\end{enumerate}
\end{theo}

\begin{proof}
$[(a) \Rightarrow (b)]$. Assume that $\rho(A+e^{i\omega}B) \geq 1$ for some $\omega \in \R$. If $\rho(A+e^{i\omega}B)>1$, then ${\cal S}(A,B)$ is unstable by Corollary \ref{unstablegeneral} and Remark~\ref{rem2}. 

Otherwise, $A+e^{i\omega}B$ has an eigenvalue $\lambda$ with $|\lambda|=1$ and a corresponding eigenvector $v\neq 0$. Define
\[ x_{i,j}=\lambda^{i+j}e^{-i\omega j}v. \]
Then
\begin{align*}
	A x_{i-1,j}+B x_{i,j-1} & = \lambda^{i+j-1}e^{-i\omega j}Av + \lambda^{i+j-1}e^{-i\omega(j-1)}Bv  \\
	& = \lambda^{i+j-1}e^{-i\omega j}(A+e^{i\omega}B)v  \\
	& = \lambda^{i+j}e^{-i\omega j}v = x_{i,j}.
\end{align*}
Thus we obtain a solution of $\mathcal S(A,B)$, with bounded boundary data, that does not converge to zero. Hence ${\cal S}(A,B)$ is not attractive, contradicting asymptotic stability.

\medskip

\noindent $[(b) \Leftrightarrow (c)]$. Clearly, \textit{(c)} implies \textit{(b)}. Conversely, let 
\[  u(z_1,z_2) = \log \rho(z_1 A + z_2 B). \]
By Vesentini's Theorem \cite{ve68}, $u$ is plurisubharmonic on $\C^2$, in the extended-valued sense. In particular, for each fixed value of one variable, it is subharmonic in the other variable. 

Now, if $|z_1|=|z_2|=1$, write $z_2 = e^{i\theta} z_1$, so that 
\[  \rho(z_1 A + z_2 B) = |z_1| \, \rho(  A + e^{i\theta} B) = \rho(A+e^{i\theta}  B) <1. \]
Hence, $u(z_1,z_2)<0$ for all $|z_1|=|z_2|=1$. 

Fix $z_2$ with $|z_2| \leq 1$. Applying the maximum principle for subharmonic functions, in the extended-valued form, to the restriction $z_1 \mapsto u(z_1,z_2)$, we obtain
\[  \max_{|z_1| \leq 1} u(z_1,z_2) = \max_{|z_1|=1} u(z_1,z_2). \]
Now fix $|z_1|=1$. Similarly, $z_2 \mapsto u(z_1,z_2)$ is subharmonic, so
\[  \max_{|z_2| \leq 1} u(z_1,z_2) = \max_{|z_2|=1} u(z_1,z_2). \]
It then follows
\[  \max_{|z_1| \leq 1, |z_2| \leq 1} u(z_1,z_2) = \max_{|z_1|=1, |z_2|=1} u(z_1,z_2). \] 
Hence, $u(z_1,z_2)<0$ for all $(z_1,z_2) \in \overline{\mathbb{D}}^2$, i.e.
\[  \rho(z_1 A + z_2 B) <1 \quad \forall (z_1,z_2) \in \overline{\mathbb{D}}^2.\]

\medskip

\noindent $[(c) \Rightarrow (d)]$. This implication is immediate: if  $\rho(z_1A+z_2B)<1$ for all $|z_1| \leq 1, |z_2| \leq 1$, then the matrix $I-(z_1A+z_2B)$ is invertible whenever $|z_1| \leq 1$, $|z_2| \leq 1$, so $\det(I-z_1A-z_2B) \neq 0$.  

\medskip

\noindent $[(d) \Rightarrow (a)]$. Assume $\det(I-z_1A-z_2B) \neq 0$ for all $(z_1,z_2) \in \overline{\mathbb{D}}^2$. Then, the matrix-valued function
\begin{equation*} 
 F(z_1,z_2) = (I-z_1A-z_2B)^{-1}
\end{equation*} 
is analytic in a neighborhood of $\overline{\mathbb{D}}^2$. In particular,  $F$ admits a power series expansion: 
\begin{equation} \label{Series}
F(z_1,z_2) = \sum_{r=0}^\infty \sum_{s=0}^\infty C_{r,s} z_1^r z_2^s
\end{equation} 
that converges absolutely for $|z_1|,|z_2| \leq  R$ for some $R>1$. We claim that $C_{r,s}=W_{r,s}(A,B)$ for every $r,s$. Indeed, for $|z_k|$ sufficiently small we have the Neumann-series expansion
\[ F(z_1,z_2)=\sum_{k=0}^\infty (z_1A+z_2B)^k. \]
Expanding each power $(z_1A+z_2B)^k$, the coefficient of the monomial $z_1^r z_2^s$ is precisely the sum of all words in the matrices $A,B$ in which $A$ appears exactly $r$ times and $B$ exactly $s$ times, namely $W_{r,s}(A,B)$.  Thus, in a neighborhood of the origin,
\[ F(z_1,z_2)= \sum_{r=0}^\infty \sum_{s=0}^\infty W_{r,s}(A,B) \,  z_1^r z_2^s. \]
By uniqueness of the Taylor expansion of an analytic function, it follows that $C_{r,s}=W_{r,s}(A,B)$ for all $r,s$. Consequently,
\begin{equation}   \label{Series-W}
	F(z_1,z_2)=  \sum_{r=0}^\infty \sum_{s=0}^\infty W_{r,s}(A,B) \,  z_1^r z_2^s
\end{equation}
throughout the closed bidisk of convergence $|z_k| \leq R$. Therefore,
\begin{equation} \label{MR}
M:=\sum_{i=0}^\infty \sum_{j=0}^\infty \| W_{i,j}(A,B) \| R^{i+j} < \infty.
\end{equation}
It follows that
\begin{equation*}
 \| W_{i,j}(A,B) \| < M R^{-(i+j)} \qquad \forall \, i,j \in \N.
\end{equation*}
Thus, if 
\[ \sup_{i\geq 1} \| x_{i,0}\| < \delta \quad \text{ and } \quad
\sup_{j\geq 1} \|x_{0,j}\| < \delta,
\]
the representation formula \eqref{BigFormEqn} gives us
\begin{align*}
 \|x_{i,j} \| & \leq  M \delta \|A \| \sum_{k=0}^{j-1} R^{-(i-1+k)}  + M \delta \| B \| \sum_{k=0}^{i-1} R^{-(j-1+k)} \\ 
  & < M \left(\| A \| R^{1-i}+ \| B\| R^{1-j} \right) \delta \, \frac{R}{R-1}
\end{align*}
which implies that $x_{i,j} \to 0$ as $i,j \to \infty$.  It follows that ${\cal S}(A,B)$ is attractive, and thus asymptotically stable by Theorem \ref{theo:att Lyap}.  
\end{proof}

\begin{rem}
	The proof above  is independent of the original argument of Fornasini and Marchesini, which relies on Huang's result concerning zeros of matrix polynomials in the closed bidisk. In particular, our proof uses the explicit representation formula for the solutions and the analytic expansion of \((I-z_1A-z_2B)^{-1}\).
\end{rem}

\begin{rem}
The asymptotic stability of $\mathcal{S}(A,B)$ in the commutative case is a simple consequence of the Fornasini--Marchesini Theorem. Indeed, if we denote by $a_k$ and $b_k$ the respective eigenvalues of $A$ and $B$ as in Theorem \ref{commutative}, then we have
\[  |a_k|+|b_k| \leq \max_{(z_1,z_2) \in \overline{\mathbb{D}}^2 } \max_{k} |z_1 a_k +z_2 b_k| = \max_{(z_1,z_2) \in \overline{\mathbb{D}}^2} \rho(z_1 A + z_2 B) < 1. \]
\end{rem}

\begin{rem} 
The proof of Theorem~\ref{bigTH} also shows that the maximum of \(\rho(z_1A+z_2B)\) over the closed bidisk is attained on the torus. More precisely,
\[ \max_{(z_1,z_2)\in \overline{\mathbb{D}}^2} \rho(z_1A+z_2B) = \max_{(z_1,z_2)\in \mathbb{T}^2} \rho(z_1A+z_2B). \]
\end{rem}

The next corollary implies that asymptotic stability is maintained if $A$ and $B$ are slightly perturbed. That is, if ${\cal S}(A,B)$ is asymptotically stable, then ${\cal S}(A',B')$ will be asymptotically stable, provided that $\|A-A'\|$ and $\|B-B'\|$ are small enough.

\begin{cor} \label{open}
For $n\ge 2$, the set of pairs $(A,B)\in \mathcal{M}_n(\C)\times\mathcal{M}_n(\C)$ for which ${\cal S}(A,B)$ is asymptotically stable is an unbounded, open, star-shaped subset of $\mathcal{M}_n(\C)\times\mathcal{M}_n(\C)$. 
\end{cor}

\begin{proof}
By Theorem~\ref{bigTH}, it suffices to prove that 
\[  \mathcal{A} := \{ (A,B) : \rho(z_1 A +z_2 B) < 1 \text{ for all } (z_1,z_2) \in \overline{\mathbb{D}}^2 \} \]
is an unbounded open, star-shaped set.  The closed bidisk $\overline{\mathbb{D}}^2$ is compact, and since the spectral radius is a continuous function of matrices, $\rho(z_1 A +z_2B)$ achieves its maximum on $|z_1|, |z_2|\leq 1$. Thus
\[ f(A,B) = \max_{|z_1|, |z_2| \leq 1} \rho(z_1A+z_2 B) =  \max_{|z_1| = |z_2| = 1} \rho(z_1A+z_2 B) \]
is continuous, and ${\cal S}(A,B)$ is asymptotically stable if and only if $f(A,B) < 1$. Since $f$ is continuous, $\mathcal{A} = f^{-1}((-\infty,1))$ is an open set.

 The homogeneity of the spectral radius implies that 
\[  \rho( z_1 (tA) + z_2 (tB)) = t \rho(z_1 A + z_2 B) <1 \]
for every $t \in [0,1]$ and $(z_1,z_2) \in \overline{\mathbb{D}}^2$. Hence,  $(tA,tB) \in \mathcal{A}$ whenever $(A,B) \in \mathcal{A}$, i.e. $\mathcal{A}$ is star-shaped with center at the origin. 

Finally,  let $N \in \mathcal{M}_n(\C)$ be a  single Jordan block with zeros along the diagonal. For any $t>0$, set $A=tN$ and $B=0$. It follows that 
\[  \rho(z_1 A + z_2 B) = \rho( tN) =0 <1, \] 
for every $(z_1,z_2) \in \overline{\mathbb{D}}^2$, so $(A,B) \in \mathcal{A}$.  But $\| A\|  = t \|N\| \to \infty$, so $\mathcal{A}$ is unbounded.
\end{proof}

We now turn our attention to the situations where $\rho(z_1 A +z_2 B) \geq 1$ for some $(z_1,z_2) \in \partial \overline{\mathbb{D}}^2$, the boundary of the closed  bidisk $\overline{\mathbb{D}}^2$. From the proof of the Fornasini--Marchesini Theorem, it suffices to consider the case $(z_1,z_2) \in \mathbb{T}^2$.  We begin by stating conditions that guarantee the instability of $\mathcal{S}(A,B)$. 

\begin{theo} \label{unstabilityTH}
Let $A,B \in \mathcal{M}_n(\C)$. Then, ${\mathcal S}(A,B)$ is unstable if one of the following conditions holds:
\begin{itemize}
  \item[(a)] $\rho(z_1A+z_2B) >1$ for some $(z_1,z_2) \in \overline{\mathbb{D}}^2$.
  \item[(b)] $\rho(z_1A+z_2B)=1$ for some $(z_1,z_2) \in \mathbb{T}^2$ provided that at least one of the unimodular eigenvalues of $z_1 A + z_2 B$ is defective. 
\end{itemize}
In particular, if the 1D system associated with $z_1A+z_2B$ is unstable for some $(z_1,z_2)\in\mathbb{T}^2$, then ${\cal S}(A,B)$ is unstable.
\end{theo}

\begin{proof}
\emph{(a)} By Remark 7, $\rho(z_1 A + z_2 B)>1$ for some $(z_1,z_2) \in \mathbb{T}^2$. The conclusion then follows as a simple consequence of Corollary \ref{unstablegeneral} and Remark \ref{rem2}.\\

\smallskip

\noindent \textit{(b)}  Let $\lambda$ be a defective unimodular eigenvalue of $z_1A+z_2B$. By Remark~\ref{rem2}, after replacing $A$ and $B$ by
$\lambda^{-1}z_1A$ and $\lambda^{-1}z_2B$, respectively, we may assume that $z_1=z_2=1$ and that $\lambda=1$.

Since $\lambda=1$ is a defective eigenvalue of $A+B$, choose a Jordan chain  $w_0\neq 0,\; w_1,\dots,w_d$, $d \geq 1$ such that
\[ (A+B-I)w_0=0, \qquad (A+B-I)w_r=w_{r-1}, \quad r=1,\dots,d. \]

Choose bounded boundary data by setting
\[ x_{0,0}=w_d \qquad x_{i,0}=x_{0,j}=w_d \qquad (i,j\geq1). \]
Since $x_{0,0}\neq0$, the generating-function identity
\eqref{U and V} takes the form
\[ (I-z_1A-z_2B)X(z_1,z_2) = x_{0,0}+U(z_1)+V(z_2). \]
For the above boundary data,
\[ U(z) = \sum_{i\geq1}(x_{i,0}-Ax_{i-1,0})z^i = \frac{z}{1-z}(I-A)w_d, \]
and similarly
\[ V(z) = \frac{z}{1-z}(I-B)w_d. \]
Hence, setting $M=A+B$,
\[ X(z,z) = (I-zM)^{-1} \left[ w_d+\frac{z}{1-z}(2I-M)w_d \right]. \]

Set
\[ N=M-I. \]
Since $w_0,\ldots,w_d$ is a Jordan chain corresponding to the eigenvalue $1$,
\[ Nw_0=0, \qquad Nw_r=w_{r-1}, \quad r=1,\ldots,d. \]
Moreover,
\[ (2I-M)w_d=(I-N)w_d=w_d-w_{d-1}. \]
Let
\[ q=(2I-M)w_d=w_d-w_{d-1}. \]
Then
\[ N^dq=w_0\neq0, \qquad N^{d+1}q=0. \]

Since
\[ I-zM=(1-z)I-zN, \]
we have, on the generalized eigenspace generated by the Jordan chain,
\[ (I-zM)^{-1}q = \sum_{k=0}^{d} \frac{z^k}{(1-z)^{k+1}}N^kq. \]
It follows from Theorem 2.2.10 in \cite{Ch12} that $(I-zM)^{-1}q$ has a pole of order exactly $d+1$ at $z=1$. On the other hand, $(I-zM)^{-1}w_d$ has a pole of order
at most $d+1$. Therefore
\[ X(z,z) = (I-zM)^{-1}w_d + \frac{z}{1-z}(I-zM)^{-1}q \]
has a pole of order exactly $d+2$ at $z=1$. More precisely,
\[ \lim_{z\to1}(1-z)^{d+2}X(z,z) = w_0\neq0. \]

Suppose, toward a contradiction, that the corresponding solution is bounded. Then there exists $C>0$ such that
\[ \|x_{i,j}\|\leq C \qquad (i,j\geq0). \]
For $0<r<1$, the series \eqref{genfun1} defining $X(r,r)$ is absolutely convergent, and 
\[ \|X(r,r)\| \leq \sum_{i,j\geq0}\|x_{i,j}\|r^{i+j} \leq C\sum_{i,j\geq0}r^{i+j} = \frac{C}{(1-r)^2}. \]
Consequently,
\[ (1-r)^{d+2}\|X(r,r)\| \leq C(1-r)^d \longrightarrow 0 \qquad (r\to1^-),\]
since $d\geq1$. This contradicts
\[ \lim_{r\to1^-}(1-r)^{d+2}X(r,r)=w_0\neq0. \]
Thus the corresponding solution is unbounded.

Finally, by linearity, the boundary data may be multiplied by an arbitrarily small positive constant while the corresponding solution remains unbounded. Hence $\mathcal S(A,B)$ is unstable.
\end{proof}

Having characterized asymptotic stability, we now turn to the borderline case in which the spectral radius reaches the unit circle. Accordingly, suppose that $A$ and $B$ satisfy
\begin{align*}
	\text{\emph{(SC1)}} \quad & \max_{(z_1,z_2)\in \overline{\mathbb{D}}^2} \rho(z_1A+z_2B)=1, \\
	\text{\emph{(SC2)}} \quad & \text{for every }(z_1,z_2)\in\mathbb{T}^2, \text{ each } \lambda\in\sigma(z_1A+z_2B) \cap \mathbb{T}	\text{ is semisimple}.
\end{align*}
Here $\sigma(z_1A+z_2B)$ denotes the spectrum of $z_1A+z_2B$. By \textup{(SC1)}, every matrix $z_1A+z_2B$, $(z_1,z_2)\in\mathbb T^2$, has spectral radius at most one, while \textup{(SC2)} excludes defective eigenvalues on the unit circle. Hence the one-dimensional system
\[ x_{n+1}=(z_1A+z_2B)x_n\]
is Lyapunov stable for every fixed $(z_1,z_2)\in\mathbb T^2$. Moreover, in the borderline regime \textup{(SC1)}, condition \textup{(SC2)} is necessary for Lyapunov stability of $\mathcal S(A,B)$ by Theorem~\ref{unstabilityTH}.

The next two examples show, however, that these one-dimensional spectral conditions are not sufficient for Lyapunov stability of the full two-dimensional system, either in the commutative or in the noncommutative setting. Thus, stability of every fixed system generated by the pencil $z_1A+z_2B$ on $\mathbb T^2$ does not by itself control the propagation of arbitrary bounded boundary data in two dimensions.

\begin{ex}\label{AIB0}
	As a simple example for which \textup{(SC1)} and \textup{(SC2)} hold and the system is neutrally stable, one may take $A=I$ and $B=0$. A noncommutative example is given by
	\[ A= \begin{bmatrix}
		0.5 & 1\\
		0 & 0
	\end{bmatrix},
	\qquad
	B= \begin{bmatrix}
		0.5 & 0\\
		0 & 0
	\end{bmatrix}. \]
	Indeed, $AB\neq BA$, and
	\[ z_1A+z_2B = \begin{bmatrix}
		z_1/2+z_2/2 & z_1\\
		0 & 0
	\end{bmatrix}.\]
	Its eigenvalues are $0$ and $(z_1+z_2)/2$. Hence
	\[ \rho(z_1A+z_2B) = \left|\frac{z_1+z_2}{2}\right| \leq 1 \qquad \text{for all }(z_1,z_2)\in\overline{\mathbb D}^{\,2}, \]
	and equality is attained, for example, at $(z_1,z_2)=(1,1)$.
	Thus \textup{(SC1)} holds. If $(z_1,z_2)\in\mathbb T^2$ and
	\[ \left|\frac{z_1+z_2}{2}\right|=1,\]
	then equality in the triangle inequality implies $z_1=z_2$. In this case the unimodular eigenvalue is distinct from the other eigenvalue $0$, and hence is semisimple. Therefore \textup{(SC2)} also holds.
	
	We now show that $\mathcal S(A,B)$ is neutrally stable. Write
	\[ x_{i,j}= \begin{bmatrix}
		u_{i,j}\\
		v_{i,j}
	\end{bmatrix}. \]
	For $i,j\geq1$, the recurrence gives
	\[ v_{i,j}=0 \]
	and
	\[ u_{i,j} = \frac12u_{i-1,j} + \frac12u_{i,j-1} + v_{i-1,j}. \]
	In particular, for $i\geq2$,
	\[ u_{i,j} = \frac12u_{i-1,j} + \frac12u_{i,j-1}.\]
	It remains only to observe that bounded boundary data for the original system give bounded boundary data for this scalar recurrence. For $i=1$,
	\[ u_{1,j} = \frac12u_{0,j}+v_{0,j} + \frac12u_{1,j-1}. \]
	Since the original boundary data are bounded, there exists $M>0$ such that
	\[ \left|\frac12u_{0,j}+v_{0,j}\right|\leq M \qquad (j\geq1). \]
	Iterating the preceding recurrence gives
	\[ u_{1,j} = \left(\frac12\right)^j u_{1,0} + \sum_{k=1}^j \left(\frac12\right)^{j-k} \left(\frac12u_{0,k}+v_{0,k}\right),\]
	and therefore
	\[ |u_{1,j}| \leq |u_{1,0}|+2M. \]
	Hence $\{u_{1,j}\}_{j\geq0}$ is bounded. On the region $i\geq 2$, the first component satisfies the scalar system $\mathcal S(0.5,0.5)$ with bounded boundary data. By Theorem~\ref{Scalar}, its solutions remain bounded. Hence $\mathcal S(A,B)$ is neutrally stable.
\end{ex}
	
\begin{ex} \label{commutingblowup}
	Consider the commuting matrices
	\[ A  = \begin{bmatrix}
		0.5 & 1 \\
		0 & 0.5
	\end{bmatrix}  \qquad \text{and} \qquad 
	B = \begin{bmatrix}
		0.5 & -1 \\
		0 & 0.5
	\end{bmatrix}. 
	\]
	The only eigenvalue of $z_1A+z_2B$ is 
	\[ \lambda(z_1,z_2) = \frac{z_1+z_2}{2}, \]
	hence 
	\[ \rho(z_1A+z_2B) \leq 1 \qquad \mbox{ for all } (z_1,z_2) \in \overline{\mathbb{D}}^2. \]
	If $|\lambda(z_1,z_2)| = 1$ with $ (z_1,z_2) \in \mathbb{T}^2$, then equality holds in the triangle inequality, and therefore $z_1 = z_2 = e^{i\theta}$ for some $\theta$. In this case 
	\[ z_1 A + z_2 B = e^{i\theta}(A+B) = e^{i\theta}I \]
	so the unimodular eigenvalue is semisimple. Thus conditions $(SC1)$ and $(SC2)$ are both satisfied. Nevertheless, by Theorem~\ref{commutative}(c), ${\cal S}(A,B)$ is unstable since the eigenvalue $0.5$ is not semisimple for either $A$ or $B$.
\end{ex}

In the preceding commutative example, Lyapunov stability fails because of defective behavior of the coefficient matrices. It is therefore natural to ask whether this is the only obstruction. More precisely, suppose that both $A$ and $B$ are diagonalizable and satisfy the conditions (SC1) and (SC2). In the commutative case these conditions, together with the diagonalizability of $A$ and $B$,  are sufficient for Lyapunov stability, by Theorem~\ref{commutative}. The following example shows that this conclusion fails in the noncommutative case.

\begin{ex}  \label{ex non-comm}
	Consider the noncommutative matrices
	\[ A=
	\begin{bmatrix}
		0 & 1\\
		0 & a
	\end{bmatrix} \qquad
	B=
	\begin{bmatrix}
		1 & 0\\
		0 & 0
	\end{bmatrix}, \qquad 0<|a|<1. \]
	Then $A$ and $B$ are diagonalizable.  The eigenvalues of $z_1A+z_2B$ are $z_2$ and $az_1$. Hence
	\[ \rho(z_1A+z_2B) = \max\{|z_2|,|a||z_1|\}\leq 1 \qquad \text{for all }(z_1,z_2)\in\overline{\mathbb D}^{\,2}, \]
	and equality is attained when $|z_2|=1$. Thus \textup{(SC1)} holds. Moreover, since $|a|<1$, the eigenvalue $az_1$ is never unimodular. If $z_2$ is unimodular, then $|az_1|\leq |a|<1$, so $z_2$ is a simple, and hence semisimple, eigenvalue. Therefore \textup{(SC2)} also holds.
	
	We now show that $\mathcal{S}(A,B)$ is nevertheless not Lyapunov stable. Write
	\[ x_{i,j}=
	\begin{bmatrix}
		u_{i,j}\\
		v_{i,j}
	\end{bmatrix}.\]
	The recurrence
	\[ x_{i,j}=Ax_{i-1,j}+Bx_{i,j-1} \]
	is equivalent to
	\[ v_{i,j}=a v_{i-1,j} \qquad u_{i,j}=u_{i,j-1}+v_{i-1,j}. \]
	Thus
	\[ v_{i,j}=a^i v_{0,j}, \]
	and consequently
	\[ u_{i,j} = u_{i,0} + \sum_{k=1}^{j} v_{i-1,k} = u_{i,0} + a^{i-1}\sum_{k=1}^{j} v_{0,k}. \]
	
	Choose bounded boundary data by setting
	\[ x_{i,0}=0 \qquad x_{0,j}= \frac{1}{j}
	\begin{bmatrix}
		0\\
		1
	\end{bmatrix} \quad (j\geq 1). \]
	These boundary data belong to $c_0$. However, for $i=1$ we obtain
	\[ u_{1,j} = \sum_{k=1}^{j} v_{0,k} = \sum_{k=1}^j\frac{1}{k}.\]
	Hence
	\[ \|x_{1,j}\|\longrightarrow \infty \qquad \text{as } j\to\infty. \]
	Therefore, bounded boundary data---or even data in $c_0$---may produce an unbounded solution. By linearity, these boundary data may be multiplied by an arbitrarily
	small positive constant while the corresponding solution remains unbounded. Hence $\mathcal S(A,B)$ is not Lyapunov stable.
\end{ex}

\begin{rem}
	In the last example, if the boundary data are compactly supported, then $v_{i,j}=0$ for all sufficiently large $j$, and the partial sums $\sum_{k=1}^j v_{0,k}$ are bounded. Hence $u_{i,j}\to 0$ as
	$i,j\to\infty$, so $x_{i,j}\rightarrow 0$ as $i,j\rightarrow\infty$.
\end{rem}

Thus, the instability mechanism in Example~\ref{ex non-comm} depends on the boundary-data class: it occurs for bounded data, and even for data in $c_0$, but disappears for compactly supported data. 

More importantly, Example~\ref{ex non-comm} shows that in the noncommutative setting, conditions \textup{(SC1)} and \textup{(SC2)} are not sufficient for Lyapunov stability, even when both coefficient matrices $A$ and $B$ are diagonalizable. The obstruction is not caused
by the unimodular spectrum of the pencil. Indeed, for
\[ (z_1,z_2)=(1,a)\in\partial\overline{\mathbb D}^{\,2}, \]
we have
\[ A+aB= \begin{bmatrix}
	a&1\\
	0&a
\end{bmatrix}, \]
so that the pencil has the defective eigenvalue $a$. Since $|a|<1$, this defect is not detected by \textup{(SC2)}. Thus, even requiring semisimplicity of every unimodular eigenvalue of $z_1A+z_2B$ on the entire boundary $\partial\overline{\mathbb D}^{\,2}$ would not exclude the instability exhibited by this example.

This suggests strengthening the spectral requirement further by asking that the entire matrix $z_1A+z_2B$ be diagonalizable for every
\[ (z_1,z_2)\in\partial\overline{\mathbb D}^{\,2}. \]
The next proposition shows, however, that this condition is much more restrictive than it may first appear: it forces $A$ and $B$ to commute, and therefore reduces the problem to the commutative case.

\begin{Prop} \label{diag-bd-poly}
Let $A,B \in \mathcal{M}_n(\C)$. Suppose that $(SC1)$ holds and  $z_1 A +z_2 B$ is diagonalizable for every $(z_1,z_2) \in \partial \overline{\mathbb{D}}^2$. Then $A$ and $B$ commute. Consequently, ${\cal S}(A,B)$ is neutrally stable. 
\end{Prop} 

\begin{proof}
For every $(\alpha, \beta) \neq (0,0)$, let $m=\max\{|\alpha|,|\beta|\}$. It follows that 
\[  \alpha A + \beta B = m ( z_1 A +z_2 B) \]
where $(z_1,z_2) \in \partial \overline{\mathbb{D}}^2$. Since $\alpha A + \beta B$ is a scalar multiple of the diagonalizable matrix $z_1 A + z_2 B$, it is also diagonalizable. Moreover, taking $(z_1,z_2)=(1,0)$ and $(0,1)$ shows that $A$ and $B$ are individually diagonalizable. By the Motzkin-Taussky Theorem \cite{MoTa55}, the matrices $A$ and $B$ must commute. Since they commute, they are simultaneously diagonalizable. The conclusion follows from Corollary~\ref{diagonal}(b). 
\end{proof}

Proposition \ref{diag-bd-poly} shows that requiring diagonalizability of the pencil $z_1A+z_2B$ on the entire boundary $ \partial \overline{\mathbb{D}}^2$ is substantially stronger than necessary, as it forces the matrices $A$ and $B$ to commute. On the other hand, Example \ref{ex non-comm} shows that the spectral conditions (SC1) and (SC2) alone are insufficient. Thus, in the noncommutative setting, the problem of finding a necessary and sufficient spectral characterization of Lyapunov stability remains open.

\subsection{Some additional stability results}

We now derive several results that provide simple sufficient conditions for the stability of the system $\mathcal{S}(A,B)$. 

\subsection*{Norm conditions}

The first one gives easily verified sufficient conditions in terms of the induced matrix norm of $A$ and $B$. 

\begin{Prop} \label{General}
Let $A, B \in \mathcal{M}_n(\C)$. ${\cal S}(A,B)$ is asymptotically stable if $\|A\| + \|B\|<1$ for some induced matrix norm. It is Lyapunov stable if $\|A\| + \|B\| \leq 1$. 
\end{Prop}

\begin{proof}
The asymptotic stability is an immediate consequence of the Fornasini--Marchesini Theorem since 
\[  \rho(z_1 A + z_2 B) \leq \| z_1 A + z_2 B \| \leq \|A\| +\|B \| <1\]
for every $(z_1, z_2) \in \overline{\mathbb{D}}^2$. 

If $\|A\|=1$ and $\|B\|=0$, or if $\|A\|=0$ and $\|B\|=1$, Lyapunov stability follows immediately from the one-dimensional case. We may therefore assume that $\|A\|<1$ and $\|B\|<1$. The triangle inequality applied to \eqref{BigFormEqn}   
\begin{align*}
 \| x_{i,j}\| & \leq  \sum_{k=0}^{j-1}  \binom{i-1+k}{k}  \|A\|^i \|B\|^k \| x_{0,j-k} \| \\   
   & + \sum_{k=0}^{i-1} \binom{j-1+k}{k} \|A\|^k \|B\|^j \|x_{i-k,0}\|.
\end{align*}
It then follows from \eqref{Scalar eq1} that 
\[  \| x_{i,j} \| = \left[ \Oo \left( \left( \frac{\|A\|}{1-\|B\|} \right)^i \right) + \Oo \left( \left( \frac{\|B\|}{1-\|A\|} \right)^j \right) \right] \delta\]
whenever $\|x_{i,0}\|,  \|x_{0,j}\| <\delta$. This bound, together with the scalar case treated in Theorem~\ref{Scalar}, immediately implies that the fixed point $x=0$ of ${\cal S}(A,B)$ is attractive
if $\|A\|+\|B\|<1$, and Lyapunov stable if $\|A\|+\|B\|\leq 1$.
\end{proof}

\medskip

\begin{ex}
No conclusion stronger than Lyapunov stability can be drawn if $\|A \|+\|B\| = 1$. Indeed, Example \ref{ex2} shows that $\mathcal{S}(A,B)$ can either be asymptotically stable or neutrally stable in that situation. Similarly, Example~\ref{ex3} shows that no general conclusion can be drawn when \(\|A\|+\|B\|>1\), since \(\mathcal S(A,B)\) may be asymptotically stable, neutrally stable, or unstable in that case.
\end{ex}

\begin{ex}
In the special case where both $A$ and $B$ are normal matrices, the matrix norm in Proposition \ref{General} may be replaced by the spectral radius. Example \ref{ex1} demonstrates that this substitution is not valid in general when either $A$ or $B$ is not normal. 
\end{ex}

\subsection*{Spectral radius conditions}

In order to obtain some sufficient conditions in terms of the spectral radius of $A+B$, it is useful to introduce the following notation. For every matrix $A= [ a_{ij}] \in \mathcal{M}_n(\C)$, we denote by $|A|$ the matrix with entries $|a_{ij}|$. By Theorem 8.1.18 in \cite{HoJo85}, we have for every $|z_1| \leq 1$, $|z_2| \leq 1$ that 
\[  \rho(z_1 A + z_2 B) \leq \rho(|z_1 A+z_2 B|) \leq \rho(|A|+|B|).\]
The next corollary is  an immediate consequence of Theorem~\ref{bigTH}.

\begin{cor} 
Let $A,B \in \mathcal{M}_n(\C)$. Then ${\mathcal S}(A,B)$ is asymptotically stable if $\rho(|A|+|B|)<1$. 
\end{cor}

\begin{ex}
The condition $\rho(|A|+|B|)<1$ is sufficient, but not necessary. Indeed, let
\[  A = \begin{bmatrix}
  0.6 & 0 \\
  0 & -0.6
\end{bmatrix} \qquad \text{and} \qquad
B = \begin{bmatrix}
  0 & 0.5 \\
  0.5 & 0
\end{bmatrix}\]
The system is asymptotically stable since
\[  \rho(z_1 A + z_2 B) = \sqrt{ |(0.6)^2 z_1^2 + (0.5)^2 z_2^2 |} \leq \sqrt{0.36+0.25}<1,\]
for every $(z_1,z_2) \in \overline{\mathbb{D}}^2$, but
\[  \rho(|A|+|B|) = 1.1>1.\] 
In general, no conclusion can be made when $\rho(|A|+|B|) \geq 1$. 
\end{ex}

\subsection*{Numerical radius conditions}

For any given matrix $A \in \mathcal{M}_n(\C)$, the numerical radius is defined as the function
\[  r(A) = \max_{ x^*x \leq 1} | x^* A x|.\]
For every matrix, it is easy to see that 
\[  \rho(A) \leq r(A) \leq \|A\|_2 \leq 2 r(A).\]
 
\begin{cor}
Let $A,B \in \mathcal{M}_n(\C)$. If $r(A)+r(B)<1$, then $\mathcal{S}(A,B)$ is asymptotically stable. 
\end{cor}

\begin{proof}
For every $|z_1|, |z_2| \leq 1$, we have 
\begin{align*}
  \rho(z_1 A + z_2 B) & \leq r(z_1 A+ z_2 B) \leq |z_1| r(A) + |z_2| r(B) < 1.
\end{align*}
Hence, ${\cal S}(A,B)$ is asymptotically stable by Theorem \ref{bigTH}.  
\end{proof}

\begin{rem}
By examples \ref{ex2} and \ref{ex3}, no conclusion regarding the stability of $\mathcal{S}(A,B)$ can be made when $r(A)+r(B) \geq 1$. 
\end{rem}

\subsection*{Joint spectral radius conditions}

We conclude this section by presenting a new sufficient condition for the stability of the fixed point $x=0$ in terms of the \emph{joint spectral radius} \cite{Ju09}, a concept introduced by Rota and Strang \cite{RoSt60}. Given matrices $A$ and $B$, the joint spectral radius $\rho(A,B)$ is defined by 
\[  \rho(A,B) = \lim_{k \to \infty} \max\{ \| M_1 M_2 \cdots M_k\|^{1/k} : M_i=A \text{ or } B \}.\]
This limit is known to exist and to be independent of the matrix norm chosen. 

\begin{Prop}
Let $A,B \in \mathcal{M}_n(\C)$. Then ${\cal S}(A,B)$ is  asymptotically stable if $\rho(A,B)<1/2$.
\end{Prop}

\begin{proof} Assume $\rho(A,B) = \rho <1/2$.  Choose $\varepsilon>0$ small enough so that $r:=\rho + \varepsilon < 1/2$.  Then, for all $i,j \in \N$, it follows from the definition of the joint spectral radius that 
\[  \| W_{i,j}(A,B) \| \leq \binom{i+j}{i} C_\varepsilon r^{i+j} \]
for some positive constant $C_\varepsilon$. Applying Proposition \ref{BigForm} and the bound in \eqref{Scalar eq1}, we obtain
\begin{align*}
  \| x_{i,j} \| &  \leq \sum_{k=0}^{j-1}\| W_{i-1,k}(A,B) \| \|A\| \|x_{0,j-k}\| + 
     \sum_{k=0}^{i-1}  \| W_{k,j-1}(A,B) \| \|B\| \|x_{i-k,0}\| \\
     & \leq  \left( \sum_{k=0}^{j-1}  \binom{i+k-1}{k}  r^{i+k-1} + \sum_{k=0}^{i-1} \binom{j+k-1}{k}  r^{j+k-1}  \right) C'_\varepsilon \delta\\
     &\leq \frac{C'_\varepsilon\delta}{1-r} \left( \left[\frac{r}{1-r}\right]^{i-1} + \left[\frac{r}{1-r}\right]^{j-1}\right).
\end{align*}     
where $C'_\varepsilon=C_\varepsilon\max\{\|A\|,\|B\|\}$.  Since $r<1/2$, the quantity $r/(1-r)<1$, and hence $\|x_{i,j}\| \to 0$ as $i,j \to \infty$. 
\end{proof}

\begin{ex}
No general stability conclusion can be drawn from the condition $\rho(A,B)=1/2$ alone. Indeed, if $A$ and $B$ commute, then their joint spectral radius satisfies
\[ \rho(A,B)=\max\{\rho(A),\rho(B)\}. \]
For the matrices in Example~\ref{ex2}, we have $\rho(A,B)=1/2$ in all three cases. Nevertheless, the corresponding fixed point $x=0$ is, respectively, asymptotically stable, neutrally stable, and unstable. Thus the borderline condition $\rho(A,B)=1/2$ does not determine the stability type of ${\cal S}(A,B)$.	
\end{ex}

\subsection*{A Lyapunov-type sufficient condition}

We next record a simple sufficient condition based on a quadratic Lyapunov function. This condition is not intended to be sharp; rather, it provides a direct Lyapunov-type criterion adapted to the local structure of the Fornasini--Marchesini recurrence
\[ x_{i,j}=Ax_{i-1,j}+Bx_{i,j-1}. \]
For every \(x\in\C^n\), let
\[ V(x)=x^*Px, \]
where \(P\) is a positive definite matrix. The natural decrement associated with the above recurrence compares the value of \(V\) at the node \((i,j)\) with the two incoming values at \((i-1,j)\) and \((i,j-1)\). Thus we consider
\begin{equation} \label{Lyapunov function}
V(x_{i,j})-V(x_{i-1,j})-V(x_{i,j-1}). 
\end{equation}
This choice is local for the two-dimensional dynamics. By contrast, a diagonal decrement such as \(V(x_{i,j})-V(x_{i-1,j-1})\) compares \(x_{i,j}\) with a state that does not determine \(x_{i,j}\) directly through the recurrence.

Assume that there exists \(0<\alpha<1/2\) such that
\begin{equation} \label{LFC}
	A^*PA+B^*PB < \alpha P.
\end{equation}
Then, for all \(u,v\in\C^n\), using the \(P\)-norm
\[ \|x\|_P^2:=x^*Px, \]
we have
\[ V(Au+Bv)=\|Au+Bv\|_P^2 \leq 2\bigl(\|Au\|_P^2+\|Bv\|_P^2\bigr). \]
Moreover, \eqref{LFC} implies
\[ \|Au\|_P^2 = u^*A^*PAu \leq \alpha u^*Pu=\alpha V(u), \]
and
\[ \|Bv\|_P^2 = v^*B^*PBv \leq \alpha v^*Pv=\alpha V(v). \]
Therefore
\[ V(Au+Bv)\leq 2\alpha\bigl(V(u)+V(v)\bigr). \]
Since \(2\alpha<1\), applying this inequality with \(u=x_{i-1,j}\) and \(v=x_{i,j-1}\) gives
\begin{equation} \label{LFD}
	V(x_{i,j})-V(x_{i-1,j})-V(x_{i,j-1})\leq 0.
\end{equation}
Thus \(V\) provides a Lyapunov-type sufficient condition for stability in the present two-dimensional setting. The condition \eqref{LFC} is conservative, but it has the advantage of being easy to verify and of leading directly to the spectral condition used in Theorem~\ref{bigTH}.

\begin{cor}
	Suppose there exists a positive definite matrix \(P\) such that \eqref{LFC} holds. Then, \(\mathcal{S}(A,B)\) is asymptotically stable.
\end{cor}

\begin{proof}
	Fix \(|z_1|,|z_2|\leq 1\) and set
	\[ M=z_1A+z_2B. \]
	For every \(x\in\C^n\), we have
	\begin{align*}
		\|Mx\|_P^2
		&= \|z_1Ax+z_2Bx\|_P^2  \\
		&\leq 2\bigl(\|z_1Ax\|_P^2+\|z_2Bx\|_P^2\bigr) \\
		&\leq 2\bigl(\|Ax\|_P^2+\|Bx\|_P^2\bigr) \\
		&= 2x^*(A^*PA+B^*PB)x \\
		&\leq 2\alpha x^*Px \\
		&= 2\alpha \|x\|_P^2.
	\end{align*}
	Hence the induced \(P\)-norm satisfies
	\[ \|M\|_P\leq \sqrt{2\alpha}<1. \]
	Consequently,
	\[ \rho(z_1A+z_2B)=\rho(M)<1 \]
	for every \(|z_1|,|z_2|\leq 1\). In particular, this holds for every \((z_1,z_2)\in\mathbb{T}^2\). By Theorem~\ref{bigTH}, \(\mathcal{S}(A,B)\) is asymptotically stable.
\end{proof}

\section{Some BIBO results}

In this section, we obtain several results about BIBO \cite{Zh17} (bounded input, bounded output) stability of the system $\mathcal{S}(A,B)$ when the boundary conditions are assumed to be in $\ell^1$. We start with a simple consequence of Proposition \ref{General}.  

\begin{Prop}
Suppose $\|A\| + \|B \| <1$. If the boundary conditions $\{ x_{i,0}\}$, $\{ x_{0,j}\}$ are in  $\ell^1(\N)$, then $\{ x_{i,j}\} \in \ell^1(\N^2)$. 
\end{Prop}

\begin{proof}
We have
\begin{align*}
  \sum_{i=1}^n \sum_{j=1}^m \| x_{i,j} \| & = \sum_{i=1}^n \sum_{j=1}^m \| A x_{i-1,j} + B x_{i,j-1} \| \\
    & \leq \|A\| \, \sum_{i=1}^n \sum_{j=1}^m \| x_{i-1,j} \| + \|B\| \, \sum_{i=1}^n \sum_{j=1}^m \|x_{i,j-1}\| \\
    & = \|A\| \, \sum_{i=0}^{n-1} \sum_{j=1}^m \|x_{i,j}\| + \|B \| \, \sum_{i=1}^n \sum_{j=0}^{m-1} \|x_{i,j} \| \\
    & \leq \|A\| \, \sum_{i=1}^{n} \sum_{j=1}^m \|x_{i,j}\| + \|B \| \, \sum_{i=1}^n \sum_{j=1}^{m} \|x_{i,j} \|  \\
    & \quad + \| A\| \sum_{j=1}^m \|x_{0,j}\| + \|B\| \, \sum_{i=1}^n \|x_{i,0} \| 
\end{align*}
Consequently,
\[  \sum_{i=1}^n \sum_{j=1}^m \| x_{i,j} \| \leq \frac{\|A\| \, \| \{x_{0,j}\} \|_1 + \|B \| \, \| \{x_{i,0}\} \|_1}{1-\|A\|-\|B\|} \]
which implies that $\{ x_{i,j} \} \in \ell^1(\N^2)$ as desired.
\end{proof}

Under the assumption that 
\[ \rho(z_1A+z_2B)\leq 1 \qquad \text{for all } (z_1,z_2)\in\mathbb T^2, \]
and that every eigenvalue \(\lambda\) of \(z_1A+z_2B\) with
\(|\lambda|=1\) is semisimple, the resolvent, for each fixed
\((z_1,z_2)\in{\mathbb T}^2\),
\[ \bigl(zI-(z_1A+z_2B)\bigr)^{-1} \]
has at most a simple pole at the eigenvalues lying on the unit circle. Hence
there exists a constant \(C(z_1,z_2)>0\) such that
\begin{equation} \label{Kreiss}
	\bigl\| (zI-(z_1A+z_2B))^{-1} \bigr\|
	\leq \frac{C(z_1,z_2)}{|z|-1},
	\qquad |z|>1.
\end{equation}
The next result shows that, if the constant in \eqref{Kreiss} can be chosen uniformly with respect to \((z_1,z_2)\in\mathbb T^2\), then \(\mathcal{S}(A,B)\) is Lyapunov stable for \(\ell^1\)-boundary data. This should be viewed primarily as a resolvent characterization, in the spirit of the classical Kreiss matrix theorem, rather than as a computationally optimal
test. Although the uniform bound in \eqref{Kreiss} may be difficult to verify directly, it provides a useful theoretical counterpart to the spectral criteria of Theorem~\ref{bigTH}.

\begin{theo}
Let $A,B \in \mathcal{M}_n(\C)$.  Suppose that the Kreiss condition holds:
\[  K:= \sup_{(z_1,z_2)\in\mathbb T^2} \sup_{|z|>1} (|z|-1)\, \left\| \bigl(zI-(z_1A+z_2B)\bigr)^{-1} \right\| <\infty.\]
Then \({\cal S}(A,B)\) is Lyapunov stable for \(\ell^1\)-boundary data.
\end{theo}

\begin{proof}
	Define
	\[ S_n(t):=\sum_{k=0}^n x_{k,n-k}e^{ikt}. \]
	By the formal generating-function identity obtained in the proof of
	Proposition 3.1, we have
	\[ X(z_1,z_2) = (I-z_1A-z_2B)^{-1}\bigl(U(z_1)+V(z_2)\bigr) \]
	as an identity of formal power series. Substituting \(z_1=e^{it}z\) and
	\(z_2=z\), and collecting coefficients according to total degree, gives
	\[ \sum_{n\ge0}S_n(t)z^n = \bigl(I-z(e^{it}A+B)\bigr)^{-1} \bigl(U(e^{it}z)+V(z)\bigr). \]
	Write
	\[ U(e^{it}z)+V(z)=\sum_{n\ge1}c_n(t)z^n, \]
	where
	\[ c_n(t) = e^{int}\bigl(x_{n,0}-Ax_{n-1,0}\bigr) + \bigl(x_{0,n}-Bx_{0,n-1}\bigr). \]
	Since the boundary data belong to \(\ell^1\), we have
	\[ \sum_{n\ge1}\|c_n(t)\| \le \sum_{n\ge1}\|x_{n,0}-Ax_{n-1,0}\| + \sum_{n\ge1}\|x_{0,n}-Bx_{0,n-1}\|. \]
	Using \(x_{0,0}=0\), it follows that
	\[ \sum_{n\ge1}\|x_{n,0}-Ax_{n-1,0}\| \le \sum_{n\ge1}\|x_{n,0}\| + \|A\|\sum_{n\ge1}\|x_{n-1,0}\| \le (1+\|A\|)\|\{x_{n,0}\}\|_1, \]
	and similarly
	\[ \sum_{n\ge1}\|x_{0,n}-Bx_{0,n-1}\| \le (1+\|B\|)\|\{x_{0,n}\}\|_1. \]
	Therefore
	\[ \sum_{n\ge1}\|c_n(t)\| \le (1+\|A\|)\|\{x_{n,0}\}\|_1 + (1+\|B\|)\|\{x_{0,n}\}\|_1, \]
	uniformly in \(t\).
	
	By the first theorem in \cite{Mo64}, the Kreiss condition implies that the powers of \(e^{it}A+B\) are uniformly bounded. Hence there exists \(M>0\), independent of \(n\) and \(t\), such that
	\[ \|(e^{it}A+B)^n\|\le M \qquad (n\ge0,\ t\in\mathbb R). \]
	Consequently, for \(|z|<1\),
	\[ \bigl(I-z(e^{it}A+B)\bigr)^{-1} = \sum_{m\ge0}(e^{it}A+B)^m z^m. \]
	Thus
	\[ \sum_{n\ge0}S_n(t)z^n = \left(\sum_{m\ge0}(e^{it}A+B)^m z^m\right) \left(\sum_{s\ge1}c_s(t)z^s\right). \]
	Comparing coefficients gives
	\[ S_n(t) = \sum_{s=1}^n (e^{it}A+B)^{n-s}c_s(t).\]
	Therefore
	\[ \|S_n(t)\| \le M\sum_{s=1}^n\|c_s(t)\| \le M\sum_{s\ge1}\|c_s(t)\|. \]
	Using the estimate above, we obtain
	\[ \|S_n(t)\| \le M\Bigl( (1+\|A\|)\|\{x_{n,0}\}\|_1 + (1+\|B\|)\|\{x_{0,n}\}\|_1 \Bigr), \]
	uniformly in \(n\) and \(t\).
	
	Finally, by Fourier inversion,
	\[ x_{k,n-k} = \frac{1}{2\pi}\int_0^{2\pi} S_n(t)e^{-ikt}\,dt. \]
	Hence
	\[ \|x_{k,n-k}\| \le \frac{1}{2\pi}\int_0^{2\pi}\|S_n(t)\|\,dt \le M\Bigl( (1+\|A\|)\|\{x_{n,0}\}\|_1 + (1+\|B\|)\|\{x_{0,n}\}\|_1 \Bigr). \]
	This estimate is uniform in \(k\) and \(n\). Therefore, given \(\varepsilon>0\),
	choosing \(\delta>0\) such that
	\[ M\bigl((1+\|A\|)+(1+\|B\|)\bigr)\delta<\varepsilon \]
	shows that
	\[ \|x_{i,j}\|<\varepsilon \]
	for all \(i,j\), whenever
	\[ \|\{x_{i,0}\}\|_1<\delta \qquad\text{and}\qquad \|\{x_{0,j}\}\|_1<\delta.\]
	Thus \({\cal S}(A,B)\) is Lyapunov stable for \(\ell^1\)-boundary data.
\end{proof}

\section{Concluding remarks and future directions}

We close by placing the preceding results in the broader context of multidimensional system theory. The examples discussed above should not be regarded as new phenomena. Rather, they illustrate a standard feature of $nD$ systems: stability of multidimensional recursions is not determined by the individual spectra of the coefficient matrices. In the present setting, the relevant object is the matrix pencil
\[ z_1A+z_2B, \qquad |z_1|=1,\quad |z_2| = 1. \]
Thus the examples are included only to connect the criterion of Theorem~\ref{bigTH} with the familiar $nD$ intuition that stability is governed by the joint spectral behavior of the family $z_1A+z_2B$, rather than by $A$ and $B$ separately.

This perspective explains several examples above. In Example~\ref{ex3}, one has $\rho(A)+\rho(B)>1$, while the system is nevertheless asymptotically stable. Hence the inequality $\rho(A)+\rho(B)>1$ is not, by itself, an obstruction to stability. Conversely, Example~\ref{ex1} shows the complementary phenomenon: $A$ and $B$ are both nilpotent, so that $\rho(A)=\rho(B)=0$, yet the corresponding system is unstable. Thus neither large values of $\rho(A)+\rho(B)$ nor small values of the individual spectral radii provide a reliable stability criterion. What matters is the spectral behavior of the full family $z_1A+z_2B$ on the closed bidisk.

Theorem~\ref{bigTH} also shows that this behavior is not merely an artifact of decoupled diagonal examples. Indeed, ${\cal S}(A,B)$ remains asymptotically stable if we replace the matrices in Example~\ref{ex3} by
\[ A=\begin{bmatrix}
	0.2 & \alpha \\
	0 & 0.5 
\end{bmatrix}
\qquad \text{and} \qquad
B=\begin{bmatrix}
	0.6 & \beta\\
	0 & 0.4
\end{bmatrix}, \]
for arbitrary complex numbers $\alpha$ and $\beta$, since $\rho(z_1A+z_2B)$ is independent of $\alpha$ and $\beta$. Thus, for generic $\alpha$ and $\beta$, the matrices $A$ and $B$ do not commute and $\rho(A)+\rho(B)>1$, yet ${\cal S}(A,B)$ is still asymptotically stable. Moreover, the matrices $A$ and $B$ may be replaced by $XAX^{-1}$ and $XBX^{-1}$ for any invertible matrix $X$ without altering stability, and by Corollary~\ref{open}, small perturbations of $A$ and $B$ do not destroy the asymptotic stability of ${\cal S}(A,B)$.

We have presented in this paper a fairly complete picture of the stability of $x=0$ for the system \eqref{S(A,B)}. One aspect that remains incomplete concerns Lyapunov stability in the noncommutative case. In that setting, we derived several sufficient conditions ensuring Lyapunov stability, as well as several sufficient conditions for instability, including criteria based on matrix norms, numerical radius, Lyapunov functions, and the joint spectral radius. However---contrary to the asymptotic stability theory developed here---we do not obtain a necessary and sufficient condition for Lyapunov stability when $A$ and $B$ do not commute. 
In particular, when the spectral conditions SC1 and SC2 hold, the system may be either neutrally stable or unstable, as illustrated by Examples~\ref{AIB0}--\ref{ex non-comm}.
Determining a complete characterization of Lyapunov stability in the noncommutative case appears to be an interesting open problem and is the principal missing ingredient for a complete stability classification of the general Fornasini--Marchesini model.

We have touched only briefly on the difference between $\ell^\infty$ and $c_0$ boundary conditions. Our focus has been on $\ell^\infty$ boundary data, for which Pringsheim convergence is the natural setting. We expect the distinction to be substantial: in the \(c_0\)-boundary framework, the borderline cases that give neutral stability in the \(\ell^\infty\)/Pringsheim setting may behave
quite differently. In particular, the spectral condition in Theorem~\ref{bigTH}(c) remains a natural sufficient condition for asymptotic stability, but it need not be necessary once the admissible boundary data are required to vanish at infinity. We intend to explore the differences between boundary conditions that are bounded and those that vanish at infinity in a future paper.

The system \eqref{S(A,B)} admits several natural extensions. A frequently studied variant incorporates an additional coupling term along the diagonal:
\begin{equation} 	\label{ABC}
	x_{i,j} = A x_{i-1,j} + B x_{i,j-1} + C x_{i-1,j-1}.
\end{equation}
The presence of the diagonal matrix $C$ fundamentally alters the structure of the dynamics. In particular, neither the closed-form representation \eqref{BigFormEqn} for $x_{i,j}$ nor the anti-diagonal recurrence in \eqref{anti} appears to extend in a straightforward manner to \eqref{ABC}. This raises a natural question: can the system \eqref{ABC} be transformed, in a suitable sense, into an instance of \eqref{S(A,B)}, or does the diagonal interaction lead to genuinely new behavior that cannot be captured by the original model?

One may also consider the more general $nD$ system \cite{Ga01} of tensors
\[ x_{i_1, i_2, \dots, i_n } = A_1 x_{i_1-1,i_2,\dots, i_n} + A_2 x_{i_1,i_2-1,\dots,i_n} + \cdots + A_n x_{i_1,i_2,\dots, i_n-1}. \]
A natural extension of Theorem~\ref{bigTH} suggests itself. We conjecture that $x=0$ is asymptotically stable if and only if
\[ \rho(z_1 A_1 + z_2 A_2 + \cdots + z_n A_n) < 1\]
for all $(z_1,\dots,z_n)$ in the $n$th torus $\mathbb{T}^n$. 

Finally, we briefly comment on the relationship with the classical 2D \emph{Roesser} model \cite{Ro75}, which is another standard state-space representation for multidimensional discrete systems. In its homogeneous form, the Roesser model propagates a horizontal state along the $i$-direction and a vertical state along the $j$-direction, leading to coupled recursions of the form
\[ x^h_{i+1,j}=A_{11}x^h_{i,j}+A_{12}x^v_{i,j},  \qquad x^v_{i,j+1}=A_{21}x^h_{i,j}+A_{22}x^v_{i,j}, \]
together with boundary data prescribed along the axes. Although this is not immediately of the form \eqref{S(A,B)}, because the two components are advanced on different indices, it can be embedded into a doubly indexed system of the type studied in this paper by working with an augmented state and an appropriate re-indexing, or lifting, of the lattice dynamics. In this way, stability questions for the Roesser model under bounded boundary data can be reduced to the analysis of an associated system $\mathcal{S}(A,B)$ on a larger state space. Consequently, the spectral criteria and sufficient conditions developed here should yield corresponding verifiable stability conditions for Roesser systems, and provide a direct bridge between the two most commonly used 2D state-space formalisms.

We intend to address these and related questions in future work.

\section*{Acknowledgments}

The authors would like to thank both anonymous referees for their careful reading of the manuscript and for several helpful comments and suggestions that improved the presentation of the paper.

\end{document}